\documentclass[12pt, letterpaper]{amsart}
\usepackage{amsfonts,amsthm,mathtools}
\usepackage[all,cmtip]{xy}
\usepackage{fullpage}
\usepackage{amsmath}
\usepackage{amssymb}
\usepackage{enumerate}
\usepackage{tikz}
\usepackage[backref]{hyperref}
\usepackage{cleveref}
\usepackage{verbatim}
\usepackage{cancel}
\usepackage{float}

\newtheorem{lem}{Lemma}[section]
\newtheorem{thm}[lem]{Theorem}
\newtheorem{proposition}[lem]{Proposition}

\theoremstyle{definition}
\newtheorem{remark}[lem]{Remark}

\DeclareMathAlphabet{\curly}{U}{rsfs}{m}{n}

\newcommand{\tors}{\operatorname{tors}}

\newcommand{\Gal}{\operatorname{Gal}}

\newcommand{\Q}{\mathbb{Q}}

\newcommand{\CC}{\mathcal{C}}
\newcommand{\Z}{\mathbb{Z}}
\newcommand{\F}{\mathbb{F}}
\newcommand{\Qbar}{\overline{\mathbb{Q}}}

\def\diam#1{\langle#1\rangle}

\DeclareMathOperator{\rk}{{rk}}

\newcommand{\PP}{{\mathbb P}}

  \newcommand{\im}{\operatorname{im}}

\newcommand{\borna}[1]{{\color{blue} \sf $\clubsuit\clubsuit\clubsuit$ Borna: [#1]}}

\newcommand{\lmfdbec}[3]{\href{https://www.lmfdb.org/EllipticCurve/Q/#1/#2/#3}{#1.#2#3}}

\mathchardef\mhyphen="2D

\title{Quadratic points on bielliptic modular curves}

\author{Filip Najman}
\address{University of Zagreb, Bijeni\v{c}ka Cesta 30, 10000 Zagreb, Croatia}
\email{fnajman@math.hr}
\urladdr{http://web.math.pmf.unizg.hr/~fnajman/}

\author{Borna Vukorepa}
\address{University of Zagreb, Bijeni\v{c}ka Cesta 30, 10000 Zagreb, Croatia}
\email{borna.vukorepa@math.hr}
\urladdr{}

\begin{document}

\begin{abstract}
Bruin and Najman \cite{BruinNajman15}, Ozman and Siksek \cite{OzmanSiksek19}, and Box \cite{Box19} described all the quadratic points on the modular curves of genus $2\leq g(X_0(n)) \leq 5$. Since all the hyperelliptic curves $X_0(n)$ are of genus $\leq 5$ and as a curve can have infinitely many quadratic points only if it is either of genus $\leq 1$, hyperelliptic or bielliptic, the question of describing the quadratic points on the bielliptic modular curves $X_0(n)$ naturally arises; this question has recently also been posed by Mazur.

We  answer Mazur's question completely and describe the quadratic points on all the bielliptic modular curves $X_0(n)$ for which this has not been done already. The values of $n$ that we deal with are $n=60,62,69,79,83,89,92,94,95,101,119$ and $131$; the curves $X_0(n)$ are of genus up to $11$. We find all the exceptional points on these curves and show that they all correspond to CM elliptic curves. The two main methods we use are Box's relative symmetric Chabauty method and an application of a moduli description of $\Q$-curves of degree $d$ with an independent isogeny of degree $m$, which reduces the problem to finding the rational points on several quotients of modular curves.
\end{abstract}

\maketitle

\section{Introduction}

An important problem in the theory of elliptic curves over number fields is to understand their possible torsion groups, parametrized by points on the modular curves $Y_1(m,n)$, and isogenies, parametrized by points on $Y_0(n)$.
After Mazur \cite{mazur78} determined the possible torsion groups over $\Q$, Kamienny, Kenku and Momose \cite{kamienny92, KM88} determined the possible torsion groups over quadratic fields. Following a pause of almost 3 decades, recent years have seen a number of advances in understanding torsion groups over number fields of degree $d$: Derickx, Etropolski, van Hoeij, Morrow and Zureick-Brown \cite{Deg3Class} determined the possible torsion groups over cubic fields and Derickx, Kamienny, Stein and Stoll \cite{DKSS} determined the primes dividing the order of all the possible torsion groups over number fields of degree $4 \leq d \leq  7$. Merel proved that the set of all possible torsion groups over all number fields of degree $d$ is finite, for any positive integer $d$ \cite{merel}. All the possible torsion groups over a fixed number field $K$, for many fixed number fields of degree 2, 3 and 4 have also been determined, see \cite{najman2010,bruin_najman2016,trbovic2020}.

Unfortunately, much less is known about possible degrees of isogenies of elliptic curves over number fields. The only number field $K$ over which the $K$-rational points on $X_0(n)$ are known for all $n$ is $\Q$, by the results of Mazur \cite{mazur77} and Kenku (see \cite{kenku1981} and the references therein) and accordingly, the only positive integer $d$ such that we know all the possible $n$ with $Y_0(n)$ having a rational point of degree $d$, is $d=1$. Even the problem of finding all $n$ such that $Y_0^+(n)(\Q)$ contains points that are neither CM nor cusps (the set of such $n$ has been conjectured by Elkies \cite{Elkies2004} to be finite), which can be considered a sub-problem of the case $d=2$, is still open.

However, lately there has been a great deal of progress in our understanding of quadratic points on $Y_0(n)$, which we will now describe, and whose further advancement is the purpose of this paper. Momose \cite[Theorem B]{momose} proved that for any fixed quadratic field $K$ which is not imaginary of class number 1, $X_0(p)(K)$ has non-cuspidal points for only finitely many $p$.
Assuming the Generalized Riemann Hypothesis, Banwait \cite{Banwait21} explicitly found, for some specific fixed number fields $K$, all the primes $p$ for which $X_0(p)(K)$ has non-cuspidal points. The first author \cite{NajmanCamb} determined all the prime degree isogenies of non-CM elliptic curves $E$ with $j(E)\in \Q$ for number fields of degree $d\leq 7$ (and conditionally on Serre's uniformity conjecture for all $d$). This has been extended to all $d>1.4\times 10^7$  unconditionally by Le Fourn and Lemos \cite[Theorem 1.3]{LeFournLemos}.

Note that the quadratic points when $Y_0(n)$ has genus $<2$ are not interesting in a sense.  When $Y_0(n)$ has genus 0, the set $Y_0(n)(K)$ is infinite for any number field $K$, while the modular curves $Y_0(n)$ with genus 1 have infinitely many quadratic points, and moreover the points do not admit a nice geometric description. 

On the other hand, for a hyperelliptic curve $X$ (we do not consider elliptic curves to be hyperelliptic; we always assume that hyperelliptic curves have genus $\geq 2$) with $J(X)$ having rank 0 over $\Q$ and with the hyperelliptic map $h:X\rightarrow \PP^1$, all but finitely many quadratic points on $X$ are pullbacks $h^{-1} (\PP^1(\Q))$ of rational points. Since in the case of $X=X_0(n)$, the hyperelliptic involution is almost always an Atkin-Lehner involution $w_d$ for some $d$ dividing $n$, it follows that all the quadratic points in $h^{-1} (\PP^1(\Q))$, and hence all but finitely many quadratic points on $X_0(n)$ correspond to $\Q$-curves of degree $d$. We call the quadratic points on $X$ that are not in $h^{-1} (\PP^1(\Q))$ \textit{exceptional} and, accordingly, the ones that are in $h^{-1} (\PP^1(\Q))$ \textit{non-exceptional} (or alternatively \textit{obvious}). Using these observations, Bruin and the first author \cite{BruinNajman15} described all the quadratic points on the hyperelliptic curves $X_0(n)$ such that $J(X)(\Q)$ is finite; of the 19 values of $n$ such that $X_0(n)$ is hyperelliptic, all but the peculiar case of $n=37$ satisfy that $J(X)(\Q)$ is finite. 

Since all the hyperelliptic $X_0(n)$ have genus $\leq 5$, the next natural step is finding all the (finitely many) quadratic points on the non-hyperelliptic modular curves $X_0(n)$ with $g(X_0(n))\leq 5$ and $\rk(J_0(n)(\Q))=0$, of which there are 15. This has been done by Ozman and Siksek \cite{OzmanSiksek19} by using the fact that the Abel-Jacobi map $\iota :X^{(2)}(\Q)\rightarrow J(X)(\Q)$ which sends $\{P,Q\}$ to $[P+Q-2P_0]$, for some fixed $P_0\in X(\Q)$, is injective, and hence all the quadratic points of $X$ can be found by checking $J(X)(\Q)$, which is, by assumption, finite. 

Box \cite{Box19} completed the description of quadratic points on $X_0(n)$ of genus $2\leq g(X_0(n)) \leq 5$ by describing the quadratic points for the 8 values of $n$ such that $\rk(J_0(n)(\Q))>0$, including the hyperelliptic case of $n=37$. In the cases $n=43,53,61$, the curves $X_0(n)$ turn out to be \textit{bielliptic} with a bielliptic map $b:X_0(n)\rightarrow X_0^+(n)$ and such that $X_0^+(n)$ is an elliptic curve of positive rank over $\Q$. As in the hyperelliptic case, it turns out that all but finitely many quadratic points are in $b^{-1}(X_0^+(n)(\Q))$ and correspond to $\Q$-curves of degree $n$. These points are, as before, called \textit{non-exceptional}, while the (finitely many) remaining points are called \textit{exceptional}. One of the main tools Box uses, and one we will make abundant use of, is the relative symmetric Chabauty method; see \Cref{sec:relsymCha}.

Recall (or see \cite{HarrisSilverman91}) that a curve $X$ defined over $\Q$ and having a $\Q$-rational point can have infinitely many quadratic points if and only if it is of gonality $\leq 2$ or if it is bielliptic with a bielliptic map $b:X\rightarrow E$ such that the elliptic curve $E$ has positive rank over $\Q$. Since the quadratic points on all hyperelliptic curves $X_0(n)$ have already been described, the next logical step towards the problem of determining all the quadratic points on all $X_0(n)$ is to study the bielliptic curves $X_0(n)$. Exactly this has been posed as a Question by Mazur \cite[Question 1 (iv)]{Mazur-quad} at the workshop \textit{Rational Points and Galois Representations} held in May 2021.

Bars \cite{Bars99} determined all the bielliptic modular curves $X_0(n)$ (there are 41 of them) and those among them with infinitely many quadratic points (28 out of the 41 satisfy this). Since many of the bielliptic curves have genus $\leq 5$, the quadratic points on all but 12 have already been described in the aforementioned papers \cite{BruinNajman15,OzmanSiksek19,Box19}. In the table below we list the remaining values $n$, their genus $g(X_0(n))$, and the rank $\rk (J_0(n)(\Q))$ of their Jacobian over $\Q$.

\label{table1}
\begin{table}[h]
	\begin{tabular}{|c|c|c||c|c|c|}
\hline
$n$ & $g(X_0(n))$ & $\rk (J_0(n)(\Q))$& $n$ & $g(X_0(n))$ & $\rk (J_0(n)(\Q))$ \\
\hline
60 & 7 & 0 & 62 & 7& 0\\
69 & 7 & 0  &  79 & 6 & 1\\
83 & 7 &  1 &  89 & 7& 1\\
92 & 10 & 1  &  94 & 11& 0 \\
95 & 9 &  0 & 101 & 8& 1\\
119 & 11 & 0  &  131 & 11&1\\

\hline
\end{tabular}
\caption{The remaining curves}
\end{table}

In this paper we describe the quadratic points on all these modular curves, answering Mazur's question completely. Furthermore, as mentioned above, this also completes the description of quadratic points on all $X_0(n)$ with infinitely many quadratic points. We explicitly find all the exceptional points and show that they correspond to CM elliptic curves and show that the non-exceptional points correspond to $\Q$-curves of degree $d_n$ for some $d_n|n$.  

Although the approach for each of the modular curves is at least a bit different than for the others, our proofs can roughly be grouped into two main methods. The first method, described in \Cref{section:qcurves}, which allows us to solve the cases $n\in \{62,69,92,$ $94 \}$ is to exploit the fact that for some of the $n$ there is a divisor $d$ of $n$ such that $X_0(d)$ is hyperelliptic and hence any putative elliptic curve $E$ with an $n$-isogeny (and hence a $d$-isogeny) over a quadratic field $K$, by the results of \cite{BruinNajman15}, has to either correspond to one of the explicitly known exceptional quadratic points on $X_0(d)$ (and we easily check whether they have an $n$-isogeny over the quadratic field they are defined over) or be a $\Q$-curve of degree $d'$ for some divisor $d'$ of $d$ for which $w_{d'}$ is the hyperelliptic involution; in all our cases we have $(d',(n/d'))=1$. In the latter case $E$ has to be a $\Q$-curve which in addition has an $(n/d')$-isogeny defined over $K$. This leads to the question: is there a modular curve whose points parametrize such elliptic curves? We give an answer to this question and show that such an elliptic curve either corresponds or is isogenous over $K$ to an elliptic curve which corresponds to a \textit{$\Q$-rational} point on one of $2$ or $3$ modular curves; see \Cref{prop:qcurves} and \Cref{prop:qcurves2} for details. The advantage of this method is that it requires very little explicit computation, as it is usually not too challenging to find all the rational points on the necessary modular curves.

The second method, used to deal with $n\in\{60,79,83,89,95,101,119,131\}$ is, following Siksek \cite{siksek} and Box \cite{Box19}, the symmetric relative Chabauty method. In principle what we do in these cases is very similar to \cite{Box19}, but there are a number of tweaks and minor improvements that we do to make the necessary computations, which have previously been done on curves of genus only up to 5, work on our curves, which are of genus $6 \leq g \leq 11$. We explain these cases in detail in \Cref{newmethod}.

The techniques used in the second method are applicable for general curves, not just modular curves. However, the information we have about cusps and automorphisms of modular curves does make some parts of the algorithms easier. In future work we plan to apply this method to other curves $X_0(n)$ of interest and more generally to other modular curves.

The code that verifies all our computations, along with the outputs containing time and memory consumption in their last line, can be found on:
\begin{center}
    \url{https://github.com/brutalni-vux/QuadPtsBielliptic}.
\end{center}

All of our computations were performed on an Intel Xeon W-2133 CPU running at 3.60GHz and with 64 Gb of RAM.

\section*{Acknowledgements}
We thank Maarten Derickx for helpful conversations and suggesting using $I=1-w_n$ in \Cref{sec:131}, Barry Mazur for helpful and motivating correspondence, Pete Clark, Tyler Genao, Paul Pollack and Frederick Saia for sharing useful data about quadratic points corresponding to CM curves over quadratic fields and Oana Adascalitei, Barinder Banwait, Abbey Bourdon, Josha Box, Andrej Dujella, Stevan Gajović and Pip Goodman for useful comments on an earlier draft. Great thanks to the referees for many helpful comments that made the paper significantly better.

We gratefully acknowledge support by the QuantiXLie Centre of Excellence, a
  project co-financed by the Croatian Government and European Union
  through the European Regional Development Fund - the Competitiveness
  and Cohesion Operational Programme (Grant KK.01.1.1.01.0004) and by
  the Croatian Science Foundation under the project
  no. IP-2018-01-1313.

%\section{Conventions, definitions and notation}
%\filip{TODO: vidjeti jel tu treba što dodati.}
%\label{section:notation}

\section{$\Q$-curves}
\label{section:qcurves}

%As any isogeny can be factored into a product of cyclic isogenies and multiplication-by-$m$, it is of course enough to consider only cyclic isogenies. 

The notion of $\Q$-curves can be defined over general number fields, but for simplicity we will restrict to $\Q$-curves over quadratic fields. For the general theory, see \cite{CremonaNajmanQCurve,Elkies2004}. An elliptic curve $E$ over a quadratic field $K$ is called a $\Q$-curve if it is isogenous (over $\overline \Q$) to its Galois conjugate. Throughout the paper, when saying that curves are isogenous, without mentioning over which field, we will always mean over $\overline \Q$. Let $\sigma$ be the generator of $\Gal(K/\Q)$. By factorising isogenies, we may assume our given isogeny $E\rightarrow E^\sigma$ is cyclic \cite[Lemma A.1]{CremonaNajmanQCurve}, and if this cyclic isogeny $E\rightarrow E^\sigma$ is of degree $d$, we say that $E$ is a $\Q$-curve of degree $d$. Throughout this section, for an integer $m$, by $C_m$ we will denote a cyclic subgroup of order $m$. Note that we allow $\Q$-curves to have complex multiplication. In particular, all elliptic curves with CM are $\Q$-curves.

We now briefly recall definitions and facts about modular curves. For a more thorough treatment, we refer the reader to \cite[IV-3]{DeligneRapoport73}. A point $x\in Y_0(n)$ represents a $\Qbar$-isomorphism class of pairs $(E,C_n)$ of an elliptic curve $E/\Qbar$ together with a cyclic subgroup of $E$ order $n$. The point $x$ is defined over a number field $K$ if and only if $\Gal_K$ acts on the $\Qbar$-isomorphism class of $E$ and on $C_n$. Equivalently, $x\in Y_0(n)$ represents a $\Qbar$-isomorphism class of pairs $(E, \mu)$ of an elliptic curve $E/\Qbar$ together with an $n$-isogeny $\mu:E\rightarrow E'$, defined over $\Qbar,$ where two isogenies $\mu:E\rightarrow E_1$ and $\mu':E'\rightarrow E_1'$  are \textit{isomorphic} if there exist $\Qbar$-isomorphism $\psi:E\rightarrow E'$ and $\varphi:E_1\rightarrow E_1'$ such that the diagram
$$
\xymatrix{
&E \ar[d]^{\mu} \ar[r]^{\psi} & E'\ar[d]^{\mu'}\\
& E_1 \ar[r]^{\varphi} & E_1'
}$$
commutes. The point $(E,C_n)$ is equal to the point $(E,\mu)$ if and only if $C_n=\ker \mu$. 

Let $S_1$ be a subgroup of $E_1$ and $S_2$ a subgroup of $E_2$, both cyclic of order $n$. We will say that these two subgroups are equal and write $S_1=S_2$ if $(E_1, S_1)$ and $(E_2,S_2)$ correspond to the same point on $X_0(n)$, or equivalently, there exists an isomorphism $\phi:E_1 \rightarrow E_2$, defined over $\overline \Q$, such that $\phi(S_1)=S_2$. Note that two isomorphic isogenies have the same kernel. Similarly, throughout the section, we will write $E_1=E_2$ if these two elliptic curves are isomorphic over $\overline \Q$, or equivalently, their $j$-invariants are equal. 

The $\Qbar$-isomorphism class of $(E,\mu)$ is defined over $K$ if both $E/\Qbar$ and $\mu$ (or equivalently $\ker \mu)$ are defined over $K$, in the sense that $\Gal_K$ acts on their $\Qbar$-isomorphism classes. Note that an elliptic curve $E/\Qbar$ is defined over $K$ if and only if $j(E)\in K$.

Let $n$ be a positive integer and factor $n=dm$ with $(d,m)=1$. Let $w_d$ be the Atkin-Lehner involution sending a point $x\in Y_0(n)$, where $x$ corresponds to $(E,C_d,C_{m})$, to the point $w_d(x)$, corresponding to $(E/C_d,E[d]/C_d,(C_{m}+C_d)/C_d)$. Here, quotienting out by $C_d$ means mapping by the $d$-isogeny $\mu$ such that $\ker \mu=C_d$, i.e. $w_d(x)=(\mu(E), \mu(E[d]),\mu(C_m))$. Thus, non-cuspidal $\Q$-rational points on $X_0(n)/w_d$ correspond to pairs
$$\left \{(E, C_d, C_m), (E/C_d,E[d]/C_d,(C_{m}+C_d)/C_d) \right\}$$
which are $\Gal_\Q$-invariant, meaning that either the point $(E, C_d, C_m)$ is defined over $\Q$ or there exists a quadratic extension $K/\Q$ with $\sigma$ generating $\Gal(K/\Q)$ such that
\begin{equation}(E, C_d, C_m)^\sigma=(E/C_d,E[d]/C_d,(C_{m}+C_d)/C_d),\label{eq:qc}\end{equation}
implying that $E$ is a $\Q$-curve of degree $d$ with the additional property that $\mu(C_{m})=C_{m}^\sigma$. We will say that an elliptic curve $F$ \textit{corresponds} to a point on $X_0(n)/w_d$ if there exists an $E$ as above such that $j(F)=j(E)$ or $j(F)=j(\mu(E))$. In the case of $d=n$, the curve $X_0(n)/w_n$ is denoted by $X_0^+(n)$ and its $\Q$-points parametrize pairs consisting of a $\Q$-curve of degree $n$ together with its Galois conjugate (without any further conditions). All the fixed points of $w_d$ correspond to CM elliptic curves.

Note that the equality in \eqref{eq:qc} is an equality of points on the modular curve $X_0(n)$, which is equivalent to the existence of an isomorphism $\phi:E^\sigma \rightarrow E/C_d$, defined over $\overline \Q$, sending $C_d^\sigma$ to $E[d]/C_d$ and $C_m^\sigma$ to $(C_{m}+C_d)/C_d$. 

\begin{remark}\label{rem:1}
To avoid any possible confusion arising from the fact that the notion of the field of definition of an object (elliptic curve, isogeny, etc.) defined over $\Qbar$ might differ from the one that the reader might be more used to, we emphasize that the field of definition of a $\Qbar$-isomorphism class of an object might be different from the field obtained by descending to a number field $K$ by taking a representative defined over $K$ of the $\Qbar$-isomorphism class and then making arguments in its $K$-isomorphism class. On one hand, if there exists a single object in the $\Qbar$-isomorphism class such that the object $A$ is itself defined over $K$, then the $\Qbar$-isomorphism class of this object is necessarily defined over $K$. Indeed, two $\Qbar$-isomorphism classes are either disjoint or equal (as is true for all equivalence classes), and since $\Gal_K$ fixes $A$, it follows that $\Gal_K$ acts on the $\Qbar$-isomorphism class of $A$. On the other hand, the $\Qbar$-isomorphism class of an object might in theory be defined over $K$, without any representative being itself defined over $K$.

To give an example of pathologies that arise when descending $\Qbar$-isomorphism classes down to a number field $K$, consider what is the field of definition of the elliptic curve $E/\Qbar$ with $j(E)=0$. Clearly if one descends down to $\Q(\sqrt 2)$ by taking the representative $E:y^2=x^3+\sqrt 2$ and asks this question, the answer is that $E$ is not defined over $\Q$ as some $\sigma \in \Gal_{\Q(\sqrt 2)}$ sends $E$ to $E':y^2=x^3-\sqrt{2}$, which is not isomorphic to $E$ over $\Q(\sqrt 2)$. However, it is easily seen that the $\Qbar$-isomorphism class of $E$ is defined over $\Q$; indeed $j(E)\in \Q$ and the curve corresponds to a $\Q$-rational point on the moduli stack $X_0(1)$.

To give a more general example, suppose $E/K$ is a $\Q$-curve of degree $d$ and $C_d=\ker \mu$ where $\mu:E\rightarrow E^\sigma$ is a $d$-isogeny, defined over $\Qbar$. 
In general $E/K$ and $E^\sigma/K$ will not be isogenous over $K$. However, by \cite[Lemma A.4]{CremonaNajmanQCurve}, there exists an isogeny $\psi:E\rightarrow (E^\sigma)^\delta$, defined over $K$, to a quadratic twist $(E^\sigma)^\delta$ of $E^\sigma$. The isogenies $\psi$ and $\mu$ are isomorphic over $\Qbar$ and hence the point $x=(E,\mu)\in Y_0(d)$ is defined over $K$. 
So although there might not exist an isogeny over $K$ between the elliptic curves $E/K$ and $E^\sigma/K$, the isogeny $\mu$ between the elliptic curves $E/\Qbar$ and $E^\sigma /\Qbar$ \textit{is} defined over $K$, as are the points $x=(E, \mu)$ and $w_d(x)=x^\sigma=(E^\sigma, \widehat{\mu})$.

\end{remark}

If $E_1/\Qbar,E_2/\Qbar,E_3/\Qbar$ are elliptic curves, $K$ is a number field and $\mu_1:E_1\rightarrow E_2$ and $\mu_2:E_2\rightarrow E_3$ are isogenies over $\Qbar$ which are defined over $K$, then $\mu_2\circ \mu_1$ is defined over $K$. Furthermore if $C_m$ is a subgroup of $E_1$ defined over $K$, then $\mu_1(C_m)$ is also defined over $K$. If $E_4/\Qbar$ and $E_5/\Qbar$ are non-CM and isogenous over $\Qbar$, then the $\Qbar$-isogeny $\mu:E_4\rightarrow E_5$ is defined over $\Q(j(E_4),j(E_5))$ by \cite[Corollary A.5]{CremonaNajmanQCurve}.

We now consider the following problem: for integers $d,m$ with $(d,m)=1$, describe  a finite collection of modular curves $\Gamma$, such that any $\Q$-curve of degree $d$ over a quadratic field with an additional $m$-isogeny gives rise to a rational point on some member of $\Gamma$. Note that a somewhat similar problem is considered in \cite[Proposition 2.2.]{Ellenberg2005}. The following propositions answer this question for $m$ prime and $m=4$, which will be sufficient for our purposes. We say that an elliptic curve has an $n$-isogeny if there exists an isogeny $\varphi:E\rightarrow E'$ with $\ker \varphi \simeq \Z/n\Z$.

\begin{proposition}
\label{prop:qcurves}
Let $E/\Qbar$ be a non-CM $\Q$-curve of degree $d$ defined over a quadratic field $K$ having in addition an $m$-isogeny defined over $K$ with $(m,d)=1$ and $m$ prime. Then either $E$ corresponds to a rational point on $X_0(dm)/w_d$ or is isogenous over $K$ to an elliptic curve which corresponds to a rational point on $X_0^+(dm^2)$.
\end{proposition}
\begin{proof}
All the elliptic curves and isogenies we will consider in  this proof will be assumed to be defined over $\Qbar$, and when saying that they are defined over $K$ we will always mean that $\Gal_K$ acts on their $\Qbar$-isomorphism class, as explained in the \Cref{rem:1}.

Suppose $E/\Qbar$ is a non-CM $\Q$-curve of degree $d$ which is defined over $K$, $C_d=\ker \mu$ where $\mu:E\rightarrow E^\sigma$ is a $d$-isogeny  and $C_{m}$ is a cyclic subgroup of order $m$ of $E$ defined over $K$. By \Cref{rem:1} it follows that $(E, C_d,C_{m})$ defines a point on $Y_0(dm)(K)$ and that $\mu$ is defined over $K$. Since $E$ is a $\Q$-curve of degree $d$, we have $E^\sigma=E/C_d$ and $C_d^\sigma=E[d]/C_d$. Now there are two possibilities: either $\mu(C_{m})=C_{m}^\sigma$ or $\mu(C_{m})\neq C_{m}^\sigma$.\\

\noindent\boxed{\mu(C_{m})=C_{m}^\sigma} In this case, by the discussion before the proposition, we see that $E$ corresponds to a rational point on $X_0(dm)/w_d$.\\

\noindent\boxed{\mu(C_{m})\neq C_{m}^\sigma} Denote by $E_1:=E/C_{m}$, by $E_2:=E^\sigma/(C_m)^\sigma$ and by $E_3:=E/(C_m+C_d)=E^\sigma/(\mu(C_m))$. Since any isogeny is a composition of a multiplication-by-$n$ map $[n]$, for some $n$, and a cyclic isogeny, and since $m$ is prime and the isogeny $f:E_2\rightarrow E_3$ is of degree $m^2$, it follows that $f$ is either $[m]$ or is cyclic.  As $\mu(C_m)\neq C_m^\sigma$, the isogenies $f_2:E_1\rightarrow E_2$ and $f_3:E_1\rightarrow E_3$ are different. Since there exists a unique (up to sign) isogeny between two non-CM elliptic curves \cite[Lemma A.1]{CremonaNajmanQCurve} it follows that $j(E_2)\neq j(E_3)$ and hence $f$ cannot be $[m]$, so we conclude that $f$ is cyclic and that $E_2$ and $E_3$ are $m^2$-isogenous. From V\'elu's formulae \cite{Velu71} it follows that the coefficients of the defining equation of $E/C_m$ are given by rational functions in the coefficients of the defining equation of $E$ and the non-zero points in $C_m$, so it follows that $(E/C_m)^\sigma=E^\sigma/C_m^\sigma $ and hence $E_1=E_2^\sigma$. Since $E_1$ and $E_3$ are $d$-isogenous, it follows that $E_1$ is a $\Q$-curve of degree $dm^2$.
\end{proof}

\begin{proposition}\label{prop:qcurves2}
Let $E/\Qbar$ be a non-CM $\Q$-curve of odd degree $d$ defined over a quadratic field $K$ having in addition an $4$-isogeny defined over $K$. Then either $E$ or a curve isogenous over $K$ to $E$ corresponds to a rational point on $X_0^+(2d)$, $X_0^+(16d)$ or $X_0(4d)/w_d$.
\end{proposition}
\begin{proof}
As in the proof of \Cref{prop:qcurves}, all the elliptic curves and isogenies we will consider here will be assumed to be defined over $\Qbar$, and when saying that they are defined over $K$ we will always mean that $\Gal_K$ acts on their $\Qbar$-isomorphism class.

Suppose $E/K$ is a non-CM $\Q$-curve of degree $d$, $C_d=\ker \mu$ where $\mu:E\rightarrow E^\sigma$ is a $d$-isogeny and $C_{4}$ is a cyclic subgroup of order $4$ of $E$ defined over $K$. By \Cref{rem:1} $(E, C_d,C_{4})$ defines a point on $Y_0(4d)(K)$ and $\mu$ is defined over $K$.\\

%Since $E$ is a $\Q$-curve of degree $d$, we have $E^\sigma=E/C_d$ and $C_d^\sigma=E[d]/C_d$. \\

\noindent\boxed{\mu(C_{4})=C_{4}^\sigma} In this case, by the discussion before \Cref{prop:qcurves}, we see that $E$ corresponds to a rational point on $X_0(4d)/w_d$.\\

\noindent\boxed{\mu(C_{4})\cap C_{4}^\sigma=\{O\}} Using the same arguments as in the $\mu(C_{m})\neq C_{m}^\sigma$ case in \Cref{prop:qcurves}, one proves that $E$ is isogenous to an elliptic curve corresponding to a point on $X_0^+(16d)$.\\

\noindent\boxed{\mu(C_{4})\cap C_{4}^\sigma=2(C_4)^\sigma=2\mu(C_4)} Let $E_1=E/(2C_4)$; it is $d$-isogenous to $E_1^\sigma$ and has all $3$ of its subgroups of order 2 defined over $K$. Indeed, as $\mu$ is defined over $K$ and of degree coprime to two, the subgroup of $E^\sigma$ generated by $\mu(C_4)$ and $C_4^\sigma$ is defined over $K$ and isomorphic to $C_2\times C_4$. The quotient of this latter group by the subgroup $\mu(C_4)\cap C_4^\sigma$ is isomorphic to $C_2 \times C_2$ and thus the two-torsion of $E_1^\sigma$ is defined over $K$. Finally, using again that $\mu$ is defined over $K$ and has degree coprime to two, we find the same is true for $E_1$. One of the subgroups of $E_1[2]$ of order $2$ is $E[2]/(2C_4)$. Call the other two $S_1$ and $S_2$. Since $\mu(C_{4})\neq C_{4}^\sigma$, it follows that $S_1^\sigma =\mu(S_2)$ and $S_2^\sigma =\mu(S_1)$. It follows $E_1/S_1$ is $4d$-isogenous to $(E_1)^\sigma/(S_1)^\sigma= (E_1)^\sigma/\mu(S_2)$, and hence corresponds to a rational point on $X_0^+(4d)$.

\end{proof}

\section{Results for $n\in\{62,69,92,94\}.$}
\label{sec:3}

We will use the following result of Momose, which we state in a weaker form, which is sufficient for our purposes.
\begin{thm}[{\cite[Theorem 0.1]{Momose1987}}]
\label{thm:momose}
Let $N$ be a composite number. If $N$ has a prime divisor $p$ such that $X_0(p)$ is of positive genus and such that $J_0(p)(\Q)$ is finite, then $X_0^+(N)$ has no non-cuspidal non-CM $\Q$-rational points.
\end{thm}
%\filip{provjeri da ako je $J_0(p)$ konačan da je i $J_0(p)^-$ isto.}

\begin{thm}
\label{prop:69}
\begin{itemize}
    \item[a)]  If $(E,C)$ is a non-cuspidal point on $X_0(62)(K)$, where $K$ is a quadratic field, then $K=\Q(\sqrt{-3})$, and either $j(E)=54000$ or $j(E)=0$ and $E$ has a point of order $2$ over $K$. 
    
    %An elliptic curve curve $E$ over a quadratic field $K$ has a $62$-isogeny if and only if $K=\Q(\sqrt{-3})$ and either $j(E)=54000$ or $j(E)=0$ and $E$ has a point of order $2$ over $K$. 

  \item[b)] If $(E,C)$ is a non-cuspidal point on $X_0(69)(K)$, where $K$ is a quadratic field, then $j(E)=-2^{15}$ and $K=\Q(\sqrt{-11})$.
  %An elliptic curve curve $E$ over a quadratic field $K$ has a $69$-isogeny if and only if $j(E)=-2^{15}$ and $K=\Q(\sqrt{-11})$.
      \item[c)] If $(E,C)$ is a non-cuspidal point on $X_0(92)(K)$, where $K$ is a quadratic field, then $K=\Q(\sqrt{-7})$ and $j(E)=-3375$ or $j(E)=16581375$.
      
      %An elliptic curve $E$ over a quadratic field $K$ has a $92$-isogeny if and only if $K=\Q(\sqrt{-7})$ and $j(E)=-3375$ or $j(E)=16581375$ .
      \item[d)]  There are no non-cuspidal quadratic points on $X_0(94)$.
\end{itemize}
\end{thm}
\begin{proof}
All the values $n$ for which we consider the modular curves $X_0(n)$ are of the form $n=mp$, where $m=2,3$ or $4$ and $p=23, 31$ or $47$. Let $K$ be a quadratic field, $(E,C)$ a $K$-rational point on $Y_0(n)$, where $E/K$ is an elliptic curve and $C$ is a $\Gal_K$-invariant cyclic subgroup of $E$ of order $n$. It follows that $y = (E, mC)$ is a $K$-rational point on $Y_0(p)$. By the results of \cite{BruinNajman15}, we know that $E$ is either a $\Q$-curve of degree $p$ or $y$ is one of the \textit{exceptional} points listed in the appropriate table in \cite{BruinNajman15}.

For each of the exceptional points $y$ listed in the appropriate table in \cite{BruinNajman15} we construct an elliptic curve with $j$-invariant $j(y)$ and determine whether it admits an $m$-isogeny. For $m=2$ this is done by checking whether the curve has a $2$-torsion point and for $m=3$ it is done by checking whether the division polynomial $\psi_3$ (see \cite[Chapter 3.2]{Washington_EC} for the definition) has a linear factor; this is a necessary and sufficient condition for the existence of a $3$-isogeny. We obtain that for $n=69$ this occurs if and only if $K=\Q(\sqrt{-11})$ and $j(y)=-2^{15}$, i.e. when the elliptic curve has complex multiplication by $\Z[\frac{1+\sqrt{-11}}{2}]$, and it does not occur for the exceptional points in the cases $n=62$ and $94$. In the remaining case $n=92$ with $m=4$, all computation can be avoided by noting that $(E,2C)$ defines a point on $Y_0(2p)(K)$, and hence $E$ is necessarily a $\Q$-curve of degree $p$ by \cite[Table 13]{BruinNajman15}.

It remains to consider the non-exceptional points, which are $\Q$-curves of degree $p$. Suppose first that $E$ does not have CM. Let first $n=62,69$ or $94$. By \Cref{prop:qcurves}, either $E$ corresponds to a rational point on $X_0(n)/w_p$ or is isogenous to a non-CM elliptic curve corresponding to a rational point on $X_0^+(pm^2)$. The latter is impossible by \Cref{thm:momose}. For $n=92$, by \Cref{prop:qcurves2} we obtain that an elliptic curve isogenous to $E$ corresponds to a rational point on $X_0(92)/w_{23}$, $X_0^+(92)$ or $X_0^+(368)$, the last two again being impossible by \Cref{thm:momose}. 
By \cite{gonz-bielliptic}, $X_0(69)/w_{23}$ is the elliptic curve \lmfdbec{69}{a}{2}, $X_0(94)/w_{47}$ is the elliptic curve \lmfdbec{94}{a}{2} and $X_0(92)/w_{23}$ is the elliptic curve \lmfdbec{92}{b}{2}. In the first two cases the elliptic curve has $2$ rational points and in the final case it has $3$ rational points, which is in all cases the same as the number of rational cusps. On the other hand, the curve $X:=X_0(62)/w_{31}$ is by \cite{gonz-bielliptic} the elliptic curve \lmfdbec{62}{a}{4}. It has $4$ rational points, 2 of which are cusps, while the remaining two correspond to elliptic curves defined with $j$-invariant 54000 and 0, which give one point each. The pullbacks of both of these two non-cuspdial rational points on $X$, with respect to the quotient map $X_0(62)\rightarrow X$, are defined over $\Q(\sqrt{-3})$. Note that for elliptic curves with $j(E)=0$ only those with a $\Q(\sqrt{-3})$-rational 2-torsion point correspond to a quadratic point on $X_0(62)$, i.e. the elliptic curves $y^2=x^3+d$ for which $d$ is a cube in $\Q(\sqrt{-3})$. 

It remains to check the existence of quadratic CM points on all the $X_0(n)$. From \cite[Corollary 8.9.c)]{CGPS21} we conclude that $X_0(n)$ has no CM points for $n=94$. Using data provided to us by the authors of \cite{CGPS21}, which can be obtained using \cite[Theorem 3.7]{CGPS21}, we find that the only quadratic CM points for the values $n$ that we haven't already found are the ones with $j(E)=-3375$ or $j(E)=16581375$ over $K=\Q(\sqrt{-7})$ for $n=92$.
\end{proof}

\section{Obtaining models for $X_0(n)$ and their quotients} \label{sec:4}
For the remaining values of $n$, it will be necessary to obtain models for $X_0(n)$ and their quotients by Atkin-Lehner involutions. We use the approach of Ozman and Siksek \cite[Section 3]{OzmanSiksek19}, but we use a different basis for the space $S_2(n)$ of weight 2 cuspforms of level $n$. %with $q$-expansion coefficients belonging to $\Q$.%

We select an Atkin-Lehner operator $w_d$, with $d \mid n$ and $d > 1$, which we will be using to get the quotient $C: = X_0(n)/w_d$. Then we choose a basis for $S_2(n)$ such that the matrix of $w_d$ is diagonal with all the diagonal elements equal to $1$ or $-1$ in that basis. This reduces the time needed to obtain a model for $X_0(n)$ and especially reduces the time needed to compute $C= X_0(n)/w_d$. A method to compute a model for $C$ that will often work in our setting, i.e. when $C$ is an elliptic curve, is to take the variables of $X_0(n)$ on which $w_d$ acts non-trivially and compute relations between them. If we succeed in obtaining a model for $C$ in this way, then the map $X_0(n)\rightarrow C$ is just the projection map.

%This makes the quotient map extremely easy to obtain. Let $g$ be the genus of $X_0(n)$ and $g'$ the genus of $X_0(n)/w_d$. 

%After ordering the basis of $S_2(n)$ so that the elements on which $w_d$ acts as multiplication by $-1$ come first, the equation for $X_0(n)/w_d$ is obtained by the equations defining $X_0(n)$ which contain only the first $g-g'$ cooordinates. The map $X_0(n)\rightarrow X_0(n)/w_d$ is just the projection map. 

We obtain the quotient map using this approach only for $n=101$ as the default \texttt{Magma} function was fast enough for all other $n$, except for $n=131$, where we use a different approach which avoids computing the quotient completely.

In addition, the models we got using our basis had shorter equations and generally smaller coefficients.

We remove the part in the Ozman-Siksek code which computes a Gröbner basis since it is only used to potentially simplify the equations and didn't seem to give us noticeable gains, while it made the computations considerably slower.

With these adjustments, we were able to obtain models for $X_0(n)$ and the quotients $X_0(n)/w_d$ quickly. For example, the computation of the model of $X_0(131)$ along with the quotient using the diagonal basis took $3.560$ seconds. On the other hand, when using the  basis of $S_2(n)$ that \texttt{Magma} returns by default, the computation of the model for $X_0(131)$ took $489.579$ seconds and it was not possible to compute the quotient map in reasonable time. The files comparing those computations can be found in our code repository. %These methods are especially important for potential future work on other modular curves, since computing the model is a crucial first step in the numerous variations of Chabauty's method.

\section{Determining the Mordell-Weil groups of $J_0(n)(\Q)$}
\subsection{Determining the ranks}

Here we will prove that the ranks of $J_0(n)(\Q)$ for our values of $n$ are as listed in Table \ref{table1}. Some of that data is already known: all the Jacobians of $X_0(n)$ of rank $0$ over $\Q$ are determined in \cite[Theorem 3.1]{Deg3Class}.

\begin{proposition}
\label{rankN}
The values of $\rk(J_0(n)(\Q))$ are as listed in Table \ref{table1}.
\end{proposition}
\begin{proof}
We can use the modular symbols package in \texttt{Magma} developed by W. Stein in \cite{stein00, stein07}.

If $L(\mathcal{A}_f, 1) \neq 0$ for some representative $f$ of some Galois orbit of Hecke eigenforms, the Kolyvagin-Logachev theorem \cite{KolyvaginLogachev89} tells us that $\rk(\mathcal{A}_f(\Q)) = 0$. Furthermore, if $L(\mathcal{A}_f, 1) = 0$ for some representative $f$ of some Galois orbit of Hecke eigenforms and the order of vanishing of $L(f, 1)$ is $1$, the Kolyvagin-Logachev theorem tells us that $\rk(\mathcal{A}_f) = [K_f : \Q]$, where $K_f$ is the Hecke eigenvalue field of $f$, which is easily computed.

All our calculations fall into one of the two aforementioned categories and by summing all $\rk(\mathcal{A}_f)$ we get $\rk(J_0(n)(\Q))$ and check that the values in the Table \ref{table1} are correct.
\end{proof}

\subsection{Determining the torsion}

Here we will describe the methods we used in our attempt to determine $J_0(n)(\Q)_{tors}$ for the values of $n$ from Table \ref{table1}. For the prime values of $n$ we will use the following result of Mazur:

\begin{thm}[{\cite[Theorem (1)]{mazur77}}]
\label{torsP}
For a prime number $p$, the number of elements in $J_0(p)(\Q)_{tors}$ is equal to the numerator of $(p-1)/12$ in minimal form and $J_0(p)(\Q)_{tors}$ is generated by the difference of two cusps.
\end{thm}

\Cref{rankN} and \Cref{torsP} now determine the Mordell-Weil group of all $J_0(p)(\Q)$ for prime $p \in \{79, 83, 89, 101, 131\}$. 

Denote by $\CC_n$ the subgroup of $J_0(n)$ generated by linear equivalence classes of differences of cusps. This subgroup is called the \textit{cuspidal subgroup} of $J_0(n)$. The \textit{rational cuspidal subgroup} is defined to be $\CC_n(\Q):=\CC_n \cap J_0(n)(\Q)$. The Manin-Drinfeld theorem states that $\CC_n(\Q)\subseteq J_0(n)(\Q)_{\tors}$. Ogg conjectured and Mazur proved (cf. \Cref{torsP}) that $\CC_n(\Q)=J_0(n)(\Q)_{\tors}$ for prime $n$. The \textit{Generalized Ogg Conjecture}, which is still open, says that $\CC_n(\Q)=J_0(n)(\Q)_{\tors}$ for all positive integers $n$. For a nice overview of the current status of the proven cases of the Generalized Ogg Conjecture, see \cite{Yoo21}.

For composite $n \in \{60, 95, 119\}$, we will use the fact that $J_0(n)(\Q)_{tors}$ injects into $J_0(n)(\F_p)$ for an odd prime $p$ of good reduction \cite[Appendix]{katz81}. 

\begin{proposition}
\label{prop:torsion}
The following holds:
\begin{itemize}
    \item[a)] $J_0(60)(\Q)_{tors} \cong \Z/4\Z \times (\Z/24\Z)^3$,
    \item[b)] $\Z/6\Z \times \Z/180\Z \leq J_0(95)(\Q)_{tors} \leq (\Z/2\Z)^2 \times \Z/6\Z \times \Z/180\Z$,
    \item[c)] $J_0(119)(\Q)_{tors} \cong \Z/8\Z \times \Z/288\Z$.
\end{itemize}
\end{proposition}

\begin{remark}
In b), by the Generalized Ogg Conjecture, we expect $J_0(95)(\Q)_{tors}$ to be equal to the lower bound, but what we prove will already be good enough for our purposes. 
\end{remark}

\begin{proof}[Proof of \Cref{prop:torsion}]
For $n=60$ we use the code of Ozman and Siksek from \cite[Section 5]{OzmanSiksek19} to get that rational cuspidal subgroup of $J_0(60)(\Q)$ is isomorphic to $\Z/4\Z \times (\Z/24\Z)^3$. We also get that $J_0(60)(\Q)_{tors}$ is isomorphic either to $\Z/4\Z \times (\Z/24\Z)^3$ or to $\Z/4\Z \times (\Z/24\Z)^2 \times \Z/48\Z$. We then compute that $J_0(60)(\F_{23})$ doesn't have an element of order $48$, so we must have $J_0(60)(\Q)_{tors} \cong \Z/4\Z \times (\Z/24\Z)^3$.

For $n$ squarefree, every cusp of $X_0(n)$ is defined over $\Q$ (this is well known in this case; more generally the field of definition of the cusps can be determined using \cite[Section 2]{LeFournLemos} for any modular curve), in particular the rational cuspidal subgroup coincides with the full cuspidal subgroup.  

For $n=95$, we compute that the cuspidal group is isomorphic to $\Z/6\Z \times \Z/180\Z$, so clearly $\Z/6\Z \times \Z/180\Z \leq J_0(95)(\Q)_{tors}$. We also get the following local information:
\begin{itemize}
    \item $J_0(95)(\F_{3}) \cong (\Z/2\Z)^2 \times \Z/60\Z \times \Z/180\Z$,
    \item $J_0(95)(\F_{7}) \cong (\Z/2\Z)^5 \times \Z/18\Z \times \Z/90\Z \times \Z/900\Z$,
    \item $J_0(95)(\F_{11}) \cong \Z/4\Z \times \Z/12\Z \times \Z/36\Z \times \Z/1658340\Z$.
\end{itemize}
Reduction modulo $7$ tells us that $J_0(95)(\Q)_{tors}$ can't have $(\Z/4\Z)^2$ as a subgroup. Reduction modulo $11$ tells us that $J_0(95)(\Q)_{tors}$ can't have $(\Z/5\Z)^2$ as a subgroup. Since we already know that $\Z/6\Z \times \Z/180\Z \leq J_0(95)(\Q)_{tors}$, we can also conclude that $J_0(95)(\Q)_{tors} \leq (\Z/2\Z)^2 \times \Z/6\Z \times \Z/180\Z$.

For $n=119$, we compute $J_0(119)(\F_{3})$ and $J_0(119)(\F_{5})$ and we get that $J_0(119)(\Q)_{tors}$ has at most $(\#J_0(119)(\F_{3}), \#J_0(119)(\F_{5})) = 2304$ elements. We can easily compute that differences of cusps generate a group isomorphic to $\Z/8\Z \times \Z/288\Z$, so we can conclude $J_0(119)(\Q)_{tors} \cong \Z/8\Z \times \Z/288\Z$.
\end{proof}

As a byproduct of our results we prove the Generalized Ogg Conjecture for some of our values of $n$.

\begin{proposition}
\label{ogg}
The Generalized Ogg Conjecture is true, i.e, $\CC_n(\Q)=J_0(n)(\Q)_{tors}$ for $n=60,62,92,94,119$.
\end{proposition}
\begin{proof}
The cases $n=60$ and $119$ have already been proved in \Cref{prop:torsion}. In all the remaining cases reduction modulo primes $p$ of good reduction is insufficient to get a sharp upper bound on the torsion subgroup, so we also apply the fact that $(T_q-\diam{q}-q)P=0$ for any $P\in J_0(n)(\Q)_{tors}$ for an odd prime $q\nmid n$, see \cite[Proposition 2.3]{DKSS}, which gives a sharp upper bound for $n=62,92,94.$
We obtain:
\begin{itemize}
    \item $J_0(62)(\Q)\simeq \Z/5\Z\times \Z/120\Z$.
    \item $J_0(92)(\Q)\simeq \Z/11\Z\times \Z/22\Z \times \Z/66 \Z$.
    \item $J_0(94)(\Q)\simeq \Z/23\Z \times \Z/92 \Z$.
\end{itemize}
\end{proof}

For $n=69$, as in the case $n=95$ already discussed in \Cref{prop:torsion}, we were unable to obtain a sharp upper bound on $J_0(n)(\Q)_{tors}$.

\section{The Relative symmetric Chabauty method}
\label{sec:relsymCha}

Our main tools for determining quadratic points on $X_0(n)$ for $n \in \{60, 79, 83, 89, 95, 101,$ $119, 131\}$ are symmetric Chabauty and relative symmetric Chabauty combined with the Mordell-Weil sieve. %The goal is to determine the points on $X^{(2)}(\mathbb{Q})$.

We will be building upon the work of Box in \cite{Box19}, which in turns builds on the work of Siksek \cite{siksek}. In \cite[Theorem 2.1.]{Box19} Box uses a criterion of Siksek for a known point of $X^{(2)}(\mathbb{Q})$ to be the only point in its residue class modulo a prime $p$ of good reduction. However, there might be infinitely many quadratic points on $X$. This will indeed happen when we have a degree $2$ map $X \rightarrow C$ and $C$ has infinitely many rational points.

To circumvent this problem, Box, again building on work of Siksek \cite{siksek}, gives a criterion for a known point of $X^{(2)}(\mathbb{Q})$ to be the only point in its residue class modulo prime $p$ of good reduction, up to pullbacks from $C(\mathbb{Q})$ in \cite[Theorem 2.4.]{Box19}. One can then, if needed, combine the information acquired from the aforementioned two theorems for different values of $p$ by using the Mordell-Weil sieve as described in \cite[Section 2.5]{Box19}.

\label{BoxIns}
The input for Box's method is: 
\begin{enumerate}
    \item a model for a non-hyperelliptic projective curve $X(\mathbb{Q})$,
    \item a set of known rational effective degree $2$ divisors on $X$;
    \item a set $\Gamma$ of matrices defining Atkin-Lehner operators on $X$ such that $C =X/\Gamma$; in all our cases $\Gamma$ will have only one element (not counting the identity), 
    \item a set of degree $0$ divisors that generate a subgroup $G$ of $J(X)(\mathbb{Q})$ of finite index;
    \item an integer $I$ such that $I \cdot J(X)(\mathbb{Q}) \subseteq G$; \label{uvjet:I}
    \item a degree $2$ effective divisor $D_{pull}$ on $X$ that is a pullback of a rational point on $C$, used to embed $X^{(2)}$ into $J(X)$.
\end{enumerate}

 For Box's method to work, the following conditions need to be satisfied:
\begin{enumerate}
    \item $\rk(J(X)(\Q)) < g(X) - 1$,
    \item $\rk(J(X)(\Q)) = \rk(J(C)(\Q))$.
\end{enumerate}

The first condition ensures that we can find at least two linearly independent differentials vanishing on $X^{(2)}(\mathbb{Q}_p) \cap \overline{J(\mathbb{Q})}$, see \cite[Section 2.2.1]{Box19}. The second condition ensures that we can find a suitable $I$ in \eqref{uvjet:I}, see \cite[Proposition 3.1.]{Box19}. It also helps us in finding annihilating differentials, see \cite[Lemma 3.4.]{Box19}. Also notice that Box proved \cite[Lemma 3.4. and Proposition 3.5.]{Box19} for the values of $n$ he considered, but it is clear that analogous proofs also work for all $n \in \{60, 79, 83, 89, 95, 101, 119, 131\}$ and for the Atkin-Lehner operators we will be using. Hence we can use the same method as Box for determining annihilating differentials.

For $Q\in X^{(2)}(\Q)$, let $\phi$ be the map sending $Q$ to $\phi(Q) = I\cdot[Q - D_{pull}] \in G$. For a $B \leq G$, $w \in G$, we call the set $w+B$ a $B$-coset represented by $w$. Suppose now that $Q\in X^{(2)}(\Q)$ is some unknown point. We start with $B_0\leq G$ and $W_0\subseteq G$ which satisfy $\bigcup_{w \in W_0}(w+B_0)= G$, (e.g. in some instances we choose $B_0:=G, W_0:=\{0\}$), from which it follows that $\phi(Q)\in \bigcup_{w \in W_0}(w+B_0)= G$. In the $i$-th step, for $i\geq 1$, after applying Chabauty and the Mordell-Weil sieve using some prime $p_i$, we create a new subgroup $B_i \leq G$ and a set $W_i$ of $B_i$-coset representatives, which satisfy $\phi(Q)\in \bigcup_{w \in W_i}(w+B_i)$. Using the information we obtained using Chabauty and the Mordell-Weil sieve, in each step we aim to shrink the set $W_i$, in the aim of getting $W_i=\emptyset$,  which would prove that there are no unknown points and hence our known points are equal to $X^{(2)}(\Q)$. For more details on how Chabauty and the Mordell-Weil sieve are applied see \cite[Section 2]{Box19} or the next section.

 We apply and, when needed, modify Box's method to describe all quadratic points on $X_0(n)$ for $n \in \{60, 79, 83, 89, 95, 101, 119, 131\}$. For composite values of $n$, we use Box's method with only a small change in how the models of $X_0(n)$ are obtained, which made some of the computations considerably faster.

 For prime values of $n < 131$, we will make some adjustments, described in the next section, that were already partially made in \cite[Sections 3 and 4]{BGG21}. For $n = 131$, we will make one substantial adjustment, which we describe in \Cref{sec:131}.

\section{Methods and computations for $n \in \{60, 79, 83, 89, 95, 101, 119, 131\}$}
\label{sec:comp}
Here we describe the methods and computations for $n \in \{60, 79, 83, 89, 95, 101, 119, 131\}$ which help us describe all quadratic points on those $X_0(n)$. For $n \in \{60, 95, 119\}$ we use Box's method as described in \Cref{sec:relsymCha} and the computations successfully determine the quadratic points on those $X_0(n)$. Therefore, in the rest of this section we will be describing the methods used for $n \in \{79, 83, 89, 101, 131\}$. The changes that we make to the method from \Cref{sec:relsymCha} will be based on Box's unpublished work on $X_0(79)$ and \cite{BGG21}. For prime values of $n$, using the same approach as in \Cref{sec:relsymCha} does not give us the desired results because we never seem to be able to get $W_i = \emptyset$. 

We improve on \cite{Box19} and Box's unpublished work by using the improved algorithms to obtain ``diagonal" (with respect to the action of $w_d$) models of $X_0(n)$, which makes our computations feasible, as explained in \Cref{sec:4}. In addition, for $n = 131$ for $I$ we use the operator $1-w'_{n}$, which seems to be a novel idea ($I$ has usually been chosen to be multiplication by an integer).

Notice that for all these values of $n$ we have $\rk(J_0(n)(\Q)) = 1$ and we have a degree $2$ quotient map $X_0(n) \rightarrow X_0^+(n)$, where $X_0^+(n)$ is an elliptic curve of rank $1$. Since Box's unpublished \texttt{Magma} file for $X_0(79)$ and the methods of \cite{BGG21} work a bit differently than the method described in \cite{Box19}, we will first describe that method, which we will call here the ``updated method" and then build upon it for larger values of $n$.

\subsection{Description of the updated method} \label{newmethod} The notation and input are the same as in \Cref{BoxIns}. Furthermore, for any object (point, divisor or divisor class) $M$ we denote by $\widetilde{M}$ the reduction of that object modulo $p$; it will always be clear from the context which prime this is. Let $X$ be a non-hyperelliptic curve of genus $g\geq 3$. For a prime $p > 2$ of good reduction for $X$, define the following mappings:
\begin{itemize}
    \item $\iota : X^{(2)}(\Q) \rightarrow J(X)(\Q)$, $\iota(P) = [P - D_{pull}]$;
    \item $\phi : X^{(2)}(\Q) \rightarrow G$, $\phi(P) = I \cdot [P - D_{pull}]$;
    \item $m : J(X)(\Q) \rightarrow G$, $m(A) = I \cdot A$;
    \item $red_p : J(X)(\Q) \rightarrow J(X)(\F_p)$, $red_p(A) =\widetilde{A}$;
    \item $h_p : G \rightarrow J(X)(\F_p)$, $h_p(A) = red_p(A)=\widetilde{A}$;
    \item $m_p : J(X)(\F_p) \rightarrow J(X)(\F_p)$, $m_p(\widetilde A) = I \cdot \widetilde A$;
\end{itemize}

Notice that the images of $m$ and $\phi$ really are in $G$ by \eqref{uvjet:I}. Also, $\iota$ is injective since $X$ is not hyperelliptic. These maps fit into a commutative diagram:
$$
\xymatrix{
&X^{(2)}(\Q) \ar[d]^{\iota} \ar[dr]^{\phi}\\
& J(X)(\Q) \ar[r]^{m} \ar[d]^{red_p} & G \ar[d]^{h_p} \\
& J(X)(\F_p) \ar[r]^{m_p} & J(X)(\F_p).}
$$

Assume there is some unknown point $Q \in X^{(2)}(\Q)$. As mentioned before, in Box's original method, described in \Cref{sec:relsymCha}, the goal was to get that $\phi(Q) \in \emptyset$, a contradiction. Here we aim to either get the same result, or, if that is not possible, obtain some information about what $Q$ has to look like (e.g. to get that $Q$ is a pullback of a rational point on $X_0^+(n)$ in our case of $X:=X_0(n)$); this additional information will allow us to solve the problem. 

 As in \Cref{sec:relsymCha}, in the $i$-th step we want to determine $B_i \leq G$ and $W_i \subseteq G$  such that $\phi(Q)$ is a member of some $B_i$-coset represented by some $w \in W_i$. We start with $i = 0$, $B_0 = G_{free}$ and $W_0 = G_{tors}$, where $G_{free}$ is the free part of $G$ and $G_{tors}$ is the torsion subgroup of $G$. Clearly, $\phi(Q)$ is a member of some $B_0$-coset represented by some $w \in W_0$, since $G = \bigcup_{w \in W_0}(w+B_0)$.

Now take some prime $p:=p_{i+1} > 2$ of good reduction for $X$. Assume we have determined $B_i$ and $W_i$ and now want to construct $B_{i+1}$ and $W_{i+1}$.  %First we shrink $B_i$ and split the cosets represented by elements in $W_i$ in such a way that $h_{p}$ is constant on each of the $B_{i+1}$-cosets represented by some $w \in W_{i+1}$. 
We set $B_{i+1} = B_i \cap \ker(h_{p})$ and split each of the $B_i$-cosets represented by elements of $W_i$ into $B_{i+1}$-cosets. We then create the set $W_{i+1}$ of representatives of $B_{i+1}$-cosets we just produced. It follows that $h_{p}$ is constant on each of the $B_{i+1}$-cosets represented by some $w \in W_{i+1}$ (in particular this was the goal of choosing this $B_{i+1}$). Clearly, $\phi(Q)$ will be in some $B_{i+1}$-coset represented by some element of $W_{i+1}$ since $\bigcup_{w \in W_i}(w+B_{i}) = \bigcup_{w \in W_{i+1}}(w+B_{i+1})$.

This splitting is useful because 
$$h_p\left(\bigcup_{w \in W_{i+1}}(w+B_{i+1})\right)=h_{p}(W_{i+1}).$$
%the image of the $B_{i+1}$-cosets represented by elements of $W_{i+1}$ under $h_{p}$ is simply $h_{p}(W_{i+1})$. 
We now apply the Mordell-Weil sieve and the Chabauty method to eliminate some elements from $W_{i+1}$. Notice that $h_{p}(\phi(Q)) \in \im(m_{p})$, so we need to only consider those $w \in W_{i+1}$ such that $h_{p}(w) \in \im(m_{p})$. Let $H_{i+1} = h_{p}(W_{i+1}) \cap \im(m_{p})$. Since we know that $h_{p}(\phi(Q)) \in H_{i+1}$, %By taking the preimage with respect to $m_{p}$, we can find all potential possibilities for $red_{p}(\iota(Q))$. 
it follows that $red_{p}(\iota(Q)) \in m_{p}^{-1}(H_{i+1})$.

We now give a criterion based on \cite[Theorem 2.1.]{Box19} which, if satisfied, will allow us to remove more elements from $W_{i+1}$. Take some $A_{p} \in m_{p}^{-1}(H_{i+1})$ and assume $red_{p}(\iota(Q)) = A_{p}$. Denote by $l(D)$ the dimension of the Riemann-Roch space of the divisor $D$. If we have $l(A_{p} + [\widetilde{D_{pull}}]) = 0$, then we get a contradiction immediately since we must have $l(A_{p} + [\widetilde{D_{pull}}]) = l([\widetilde{Q}]) > 0$ since $\widetilde{Q}$ is an effective degree $2$ divisor. If one of our known points $Q_{known} \in X^{(2)}(\Q)$ satisfies $red_{p}(\iota(Q_{known})) = A_{p}$ and fulfills the criterion given by \cite[Theorem 2.1.]{Box19} then $(red_{p} \circ \iota)^{-1}(A_{p})=\{Q_{known}\}$. Notice that $Q \neq Q_{known}$ because $Q$ is an unknown point.
%\cite[Theorem 2.1.]{Box19} only gives us a criterion for $Q_{known}$ to be the only point in its residue disc modulo $p$.
If the criterion succeeds, that means that $Q_{known}$ is the only point in its residue disc modulo $p$. If we had $[\widetilde{Q_{known}} - \widetilde{D_{pull}}] = [\widetilde{Q} - \widetilde{D_{pull}}]$, that would imply $\widetilde{Q} = \widetilde{Q_{known}}$ since $X$ is non-hyperelliptic, which is a contradiction. Hence, if the criterion succeeds, we cannot have $red_{p}(\iota(Q)) = A_{p}$ for an unknown $Q$.

Clearly, if $w_h \in H_{i+1}$ and if $red_{p}(\iota(Q))$ can't equal any of the elements of $m_{p}^{-1}(w_h)$, then $h_{p}(\phi(Q)) \neq w_h$, hence $\phi(Q) \notin h_{p}^{-1}(w_h)$. That means we can remove the elements of $h_{p}^{-1}(w_h)$ from $W_{i+1}$  while still having $\phi(Q) \in \bigcup_{w \in W_{i+1}}(w+B_{i+1})$ satisfied.

To recapitulate, we have obtained $B_{i+1} \leq G$ and a set $W_{i+1}$ of $B_{i+1}$-coset representatives such that $\phi(Q)\in \bigcup_{w \in W_{i+1}}(w+B_{i+1})$ and
$$ \bigcup_{w \in W_{i+1}}(w+B_{i+1})\subseteq \bigcup_{w \in W_i}(w+B_{i}).$$

By repeating this for various primes $p$, it would be ideal to get $W_s = \emptyset$ for some $s$, which would imply that there are no unknown points in $X^{(2)}(\Q)$.

Unfortunately, we are unable to get $W_s = \emptyset$ when $X = X_0(n)$ and $n\in\{79,83,89,101,131\}$; this is of course not surprising as $X^{(2)}(\Q)$ is infinite in these cases. However, we get that for some $s$ both $B_s$ and $W_s$ contain only elements of the form $aD_n$, where $D_n$ is a pullback of a generator of $J_0^+(n)(\Q)$ that generates the free part of $G$. That means that for any $Q\in X^{(2)}(\Q)$ we have $\phi(Q) = I \cdot [Q - D_{pull}] = aD_n$, which will be very useful to us, as will be explained in more detail in the next subsections. Note that in the method described in this subsection we use only \cite[Theorem 2.1]{Box19} and do not use \cite[Theorem 2.4]{Box19}.

\subsection{Selecting $G$ and $I$} \label{GandI} We now describe how to select an appropriate $G$ and $I$ (see also \cite[Section 3.3.]{Box19}). Let $\rho_n : X_0(n) \rightarrow X_0^+(n)$ be the degree $2$ quotient map we get from $w_n$ and $(\rho_n)_{*}:J_0(n)\rightarrow  J_0^+(n)$ the induced (pushforward) map on $J_0(n)$. We have the following commutative diagram:

$$
\xymatrix{
& X_0(n) \ar[r]^{\iota_1} \ar[d]^{\rho_n} & J_0(n) \ar[d]^{(\rho_{n})_*} \\
& X_0^+(n) \ar[r]^{\iota_2} & J_0^+(n).}
$$

For all $n \in \{79, 83, 89, 101, 131\}$ we have that $J_0^+(n)(\Q)$ is an elliptic curve with Mordell-Weil group isomorphic to $\Z$. Let $P_n$ be a generator of $J_0^+(n)(\Q)$ and set $D_n = ((\rho_n)_*)^{*}(P_n)$. Let $T_n \in J_0(n)(\Q)$ be the divisor class of the difference of the two cusps, which is a generator of $J_0(n)(\Q)_{tors}$ (see \Cref{torsP}).

Suppose now $n\neq 131$, as for $n=131$ we will select different $G$ and $I$, see \Cref{sec:131}. Set $G = \langle D_n, T_n \rangle$. Now we can use \cite[Proposition 3.1.]{Box19} to conclude that $2 \cdot J_0(n)(\Q) \subseteq G$ so we can use $I = 2$. Let $w_n':J_0(n)\rightarrow J_0(n)$ be the map induced by $w_n$. Notice that $w_n'(D_n) = D_n$ since $D_n$ is a pullback and that $w_n'(T_n) = -T_n$ since $w_n$ swaps the cusps.

\begin{lem} \label{oneMinusW}
Let $n \in \{79, 83, 89, 101, 131\}$. Then for every $D \in J_0(n)(\Q)$ we have $(1 - w_n')(D) \in J_0(n)(\Q)_{tors}$.
\end{lem}
\begin{proof}
%Assume $D \in J_0(n)(\Q)$ is such that $(1 - w_n')(D)$ is of infinite order. Then $(1 - w_n')(2D)$ is also of infinite order. 
By the information above, we know that $2D \in \langle D_n, T_n \rangle$, so $2D = aD_n + bT_n$. Hence $(1 - w_n')(2D) = (1 - w_n')(aD_n + bT_n) = 2bT_n$, which is of finite order. 
\end{proof}

 %With that information, we can prove the following lemma:

\begin{comment}
\begin{lem} \label{oneMinusWGeneral}
\borna{općenita varijanta ako se odlučimo iz nekog razloga za nju}
Let $n \in \mathbb{N}$ and assume that $\rk(J_0(n)(\Q)) = \rk(J_0^+(n)(\Q))$. Then for every $D \in J_0(n)(\Q)$ we have $(1-w_n')(D) \in J_0(n)(\Q)_{tors}$.
\end{lem}
\begin{proof}
Similarly to the proof of \cite[Proposition 3.1.]{Box19}, let $P_1, \ldots, P_r$ be the linearly independent generators of $J_0^+(n)(\Q)/J_0^+(n)(\Q)_{tors}$ and for each $i \in \{1, \ldots, r\}$ set $D_i = \rho_n^{*}(P_i)$. Notice that $w_n'(D_i) = D_i$ because each $D_i$ is a pullback. Assume $D \in J_0(n)(\Q)$ such that $(1-w_n')(D)$ is of infinite order. Then $(1-w_n')(2D)$ is also of infinite order. By \cite[Proposition 3.1.]{Box19}, we know that $2D = \sum_{i = 1}^{r}a_iD_i + T$, where $T \in J_0(n)(\Q)_{tors}$. Then we easily see that $(1-w_n')(2D) = T - w_n'(T)$, which is of finite order, a contradiction.
\end{proof}
\end{comment}

\subsection{Quadratic points on $X_0(n)$ for $n \in\{79, 83, 101\}$} \label{easyP} 
Before proceeding any further, we mention again that $n=79$ was solved completely by Box in an unpublished \texttt{Magma} file. By performing the computations described in \Cref{newmethod}, we get that if there is an unknown $Q \in X_0(n)^{(2)}(\Q)$, then $\phi(Q) = 2 \cdot [Q - D_{pull}] = kD_n$ for some integer $k$. Hence $w_n'(2 \cdot [Q - D_{pull}]) = 2 \cdot [Q - D_{pull}]$ which means that $w_n'([Q - D_{pull}]) - [Q - D_{pull}]$ is of order at most $2$. Since $J_0(n)(\Q)$ doesn't have an element of order $2$ for $n \in \{79, 83, 101\}$, we have $w_n'([Q - D_{pull}]) = [Q - D_{pull}]$. Since $D_{pull}$ is a pullback and $X_0(n)$ is not hyperelliptic, we get $w_n(Q) = Q$. Since $Q = \{Q_1, Q_2\}$ (as a $2$-set), where $Q_i \in X_0(n)(\overline \Q)$, we conclude that either $w_n$ swaps $Q_1$ and $Q_2$ or it fixes them both. In the first case, $Q$ is a pullback of a rational point on $X_0^+(n)$. To deal with the second case, we compute the fixed points of $w_n$, which all correspond to CM curves (note that the fixed points could also have been determined using the methods of \cite{CGPS21}).

\subsection{Quadratic points on $X_0(89)$} \label{sec:89} Notice that $J_0(89)(\Q) \cong \Z \times \Z/22\Z$ has an element of order $2$, so the approach from \Cref{easyP} won't work without modification. We need more information and we get it by inspecting the possibilities for $\phi(Q)$ a bit more closely. By performing the computations from \Cref{newmethod}, we get that if there is an unknown $Q \in X_0(89)^{(2)}(\Q)$, then $\phi(Q) = 2 \cdot [Q - D_{pull}] = 2kD_{89}$ for some integer $k$. Hence $[Q - D_{pull}] - kD_{89}$ is of order at most $2$. If we have $[Q - D_{pull}] = kD_{89}$, then $w_n'([Q - D_{pull}]) = [Q - D_{pull}]$ and we continue as in \Cref{easyP}. The other possibility is that $[Q - D_{pull}] = kD_{89} + 11T_{89}$, but since $w_n'(T_{89}) = -T_{89}$ and $11T_{89} = -11T_{89}$, we would have $w_n'([Q - D_{pull}]) = [Q - D_{pull}]$ and again we can continue as in \Cref{easyP}. 

\subsection{Quadratic points on $X_0(131)$} \label{sec:131}

%For $n = 131$ we were not able to compute $\rho_{131} : X_0(131) \rightarrow X_0^+(131)$ so we were unable to determine $D_{131}$. Therefore, we will not be able to use $G=\langle D_{131}, T_{131} \rangle$ and $I=2$. 

In this case we use a different approach; in particular, we will take $I$ to be the operator $1-w'_{131}$. The reason for doing this is that we had originally been unable to compute the quotient curve and hence could not proceed as in the other cases. Although we now  know how to compute the quotient map, we have left this case as it was. This is because it is possible that the strategy of applying the operator $I=1-w'_{131}$ might be useful in future applications. An additional benefit of using this operator $I$ is that we do not need to choose a finite index subgroup for $G$. We set $G = \langle T_{131} \rangle = J_0(131)(\Q)_{tors}$, where $T_{131}$ is the divisor class of the difference of the two cusps of $X_0(131)$ as before, and $I = 1 - w_{131}'$. By Lemma \ref{oneMinusW}, we know that $I \cdot J_0(131)(\Q) \subseteq G$. Since the action of $w_{131}'$ commutes with the reduction modulo a prime $p$ of good reduction, we now proceed as in \Cref{newmethod}.

After doing the computations described in \Cref{newmethod}, we get that if there is an unknown $Q \in X_0(131)^{(2)}(\Q)$, then $\phi(Q) = (1 - w_{131}') \cdot [Q - D_{pull}] = 0$, so $w_{131}'([Q - D_{pull}]) = [Q - D_{pull}]$. Now we can again continue as in \Cref{easyP}.

\section{Example of the sieving process for $n = 89$}
Here we describe the whole process for determining the quadratic points on $X_0(89)$ in more detail. As mentioned before, we know that $J_0(89)(\Q) \cong \Z \times \Z/22\Z$. Following the notation and method from \Cref{GandI}, we are able to find a finite index subgroup $G \leq J_0(89)(\Q)$, where $G = \langle D_{89}, T_{89} \rangle$. The group is of index at most $2$, so we use $I = 2$. We will use the primes $3$, $5$ and $7$ for the sieving method described in \Cref{newmethod}. We have $B_0 = \langle D_{89} \rangle$ and $W_0 = \{ k \cdot T_{89} : 0 \leq k \leq 21 \}$. After performing the sieve with $p = 3$, we get $B_1 = \langle 5D_{89} \rangle$ and $W_1 = \{0, D_{89}, 2D_{89}, -D_{89}, -2D_{89}, 2T_{89}, 20T_{89}\}$. After performing the sieve with $p = 5$, we get $B_2 = \langle 35D_{89} \rangle$ and $W_2 = \{0, k \cdot D_{89}: 0 \leq k \leq 34 \}$. After performing the sieve with $p = 7$, we get $B_3 = \langle 210D_{89} \rangle$ and $W_3 = \{0, 2k \cdot D_{89}: 0 \leq k \leq 104 \}$. That means that $\phi(Q) = 2 \cdot [Q - D_{pull}] = 2kD_{89}$ for an unknown $Q \in X_0(89)^{(2)}(\Q)$ and we can continue as described in \Cref{sec:89}.

This example is instructive because it shows that even if we sometimes don't reach an empty set with the sieve, we might be able to reach a set with a property which enables us to proceed.

\section{Final results and tables}
 Here we describe the results that we get for $n\in \{60, 79, 83, 89, 95, 101, 119, 131\}$. Recall that the results for $n\in\{62,69,92,94\}$ can be found in \Cref{sec:3}. We describe all quadratic points on those $X_0(n)$. By $C$ we always denote the quotient of $X_0(n)$ by an Atkin–Lehner involution which we used while doing the appropriate computations described in \Cref{sec:comp}. As usual, for elliptic curves $O$ always denotes the point at infinity. The column denoted by CM lists the discriminant of the order by which the elliptic curve has complex multiplication if it does, and NO if it does not have CM. 
 
 For $n\in\{62,69,92,94\}$ we did not need to do almost any computations and did not even need to compute a model for $X_0(n)$. Hence we do not display any information for these values of $n$.
 
 We list for each $n$ the elliptic curve $C$ which we use in our computations, where $b:X_0(n)\rightarrow C$ is of degree 2. We call the quadratic points on $X_0(n)$ which no not lie in $b^{-1}(C(\Q))$ \textit{exceptional}. Note that since a bielliptic curve might have multiple maps of degree 2 to elliptic curves, whether a point is exceptional depends on the choice of $C$ (or equivalently $b$). The points we list are exceptional with respect to our choice of $C$.
 
\subsection{$X_0(60)$}
Model for $X_0(60)$: \begin{align*} 
&x_0^2 + 2x_1^2 - x_2^2 + 6x_3^2 - 6x_3x_5 + x_4^2 + 4x_5^2 - x_6^2 = 0,\\
&x_0x_2 - x_1^2 + x_3^2 + x_3x_5 = 0,\\
&x_0x_3 - x_1x_2 = 0,\\
&x_0x_4 - x_2^2 - x_3^2 - x_3x_5 + x_5^2 = 0,\\
&x_0x_5 - x_2x_3 - x_3x_4 = 0,\\
&x_1x_3 - x_2^2 - x_3x_5 = 0,\\
&x_1x_4 - x_2x_3 - x_3x_4 = 0,\\
&x_1x_5 - x_3^2 - x_3x_5 = 0,\\
&x_2x_4 - x_3^2 - x_3x_5 + x_5^2 = 0,\\
&x_2x_5 - x_3x_4 = 0.\\
\end{align*}

\noindent Genus of $X_0(60)$: $7$.

\noindent Cusps:
$(0 : 0 : 0 : 0 : -1 : 0 : 1),
(0 : 0 : 0 : 0 : 1 : 0 : 1),
(0 : 0 : -1/4 : -1/4 : 1/4 : 1/4 : 1),
(0 : 0 : -1/4 : 1/4 : 1/4 : -1/4 : 1),
(0 : 0 : 1/4 : -1/4 : -1/4 : 1/4 : 1),
(0 : 0 : 1/4 : 1/4 : -1/4 : -1/4 : 1),
(-1 : 0 : 0 : 0 : 0 : 0 : 1),
(-1/2 : -1/2 : -1/4 : -1/4 : -1/4 : -1/4 : 1),
(-1/2 : 1/2 : -1/4 : 1/4 : -1/4 : 1/4 : 1),
(1/2 : -1/2 : 1/4 : -1/4 : 1/4 : -1/4 : 1),
(1/2 : 1/2 : 1/4 : 1/4 : 1/4 : 1/4 : 1),
(1 : 0 : 0 : 0 : 0 : 0 : 1)$.

\noindent $C = X_0(60)/w_{15}$: elliptic curve $y^2 = x^3 + x^2 - x$.

\noindent Group structure of $J(C)(\Q)$: $J(C)(\Q) \simeq  \Z/6\Z \cdot [Q_C - O]$, where $Q_C := (1,-1)$.

\noindent Group structure of $G \subseteq J_0(60)(\Q)$: $G \simeq  \Z/4\Z \cdot D_1 \oplus \Z/24\Z \cdot D_2 \oplus \Z/24\Z \cdot D_3 \oplus \Z/24\Z \cdot D_4$, where $D_1, \ldots, D_4$ are generated by differences of cusps.

\noindent There are \textbf{no} quadratic points on $X_0(60)$ apart from cusps which are all defined over $\Q$.

\noindent Primes used in sieve: $13$.

\subsection{$X_0(79)$}
Model for $X_0(79)$: \begin{align*} 
&x_0^2 - 2x_0x_1 - x_1^2 - 3x_2^2 - 2x_2x_3 + 4x_2x_4 + 3x_3^2 + 2x_3x_4 - 11x_4^2 - x_5^2  = 0,\\
&x_0x_2 - x_1^2 + 3x_2x_3 - x_2x_4 - 2x_3^2 - x_3x_4 + 4x_4^2  = 0,\\
&x_0x_3 - x_1x_2 + 2x_2x_3 - x_3^2 + x_4^2  = 0,\\
&x_0x_4 - x_2^2 + 2x_2x_3 + x_2x_4 - x_3^2 - x_3x_4 + 2x_4^2  = 0,\\
&x_1x_3 - x_2^2 + 2x_2x_4 - x_4^2  = 0,\\
&x_1x_4 - x_2x_3 + x_2x_4 + x_3^2 - 2x_4^2 = 0.\\
\end{align*}

\noindent Genus of $X_0(79)$: $6$.

\noindent Cusps:
$(1 : 0 : 0 : 0 : 0 : 1), (-1 : 0 : 0 : 0 : 0 : 1)$.

\noindent $C = X_0^+(79)$: elliptic curve $y^2 + xy + y = x^3 + x^2 - 2x$.

\noindent Group structure of $J(C)(\Q)$: $J(C)(\Q) \simeq \Z \cdot [Q_C - O]$, where $Q_C := (0,0)$.

\noindent Group structure of $G \subseteq J_0(79)(\Q)$: $G \simeq  \Z \cdot D_{79} \oplus \Z/13\Z \cdot T_{79}$, where $D_{79}$, $T_{79}$ are as in \Cref{GandI}.

\noindent There are \textbf{no} exceptional non-cuspidal quadratic points on $X_0(79)$.

\noindent Primes used in sieve: $3$, $5$.

\subsection{$X_0(83)$}
Model for $X_0(83)$: \begin{align*} 
&x_0^2 - 2x_0x_1 - x_1^2 - x_2^2 + 2x_2x_3 - x_3^2 - 4x_3x_4 + 2x_3x_5 - 10x_4^2 + 24x_4x_5 - 31x_5^2 - x_6^2 = 0,\\
&x_0x_2 - x_1^2 + 2x_2x_3 - 2x_3^2 + 6x_3x_4 - 6x_3x_5 - 3x_4^2 + 6x_4x_5 = 0,\\
&x_0x_3 - x_1x_2 + 2x_2x_3 - 2x_3^2 + 5x_3x_4 - 2x_3x_5 - 4x_4^2 + 8x_4x_5 - 6x_5^2 = 0,\\
&x_0x_4 - x_2^2 + x_2x_3 + x_3x_4 - 3x_3x_5 + 4x_4^2 - 6x_4x_5 + 9x_5^2 = 0,\\
&x_0x_5 - x_3^2 + 2x_3x_4 - x_4^2 + 2x_4x_5 = 0,\\
&x_1x_3 - x_2^2 + 2x_3^2 - 4x_3x_4 - x_3x_5 + 7x_4^2 - 14x_4x_5 + 15x_5^2 = 0,\\
&x_1x_4 - x_2x_3 + x_3^2 - 2x_3x_4 + 2x_3x_5 + 3x_4^2 - 8x_4x_5 + 6x_5^2 = 0,\\
&x_1x_5 - x_3x_4 + x_3x_5 + x_4^2 - 2x_4x_5 = 0,\\
&x_2x_4 - x_3^2 + x_3x_4 + x_3x_5 - 3x_4^2 + 5x_4x_5 - 6x_5^2 = 0,\\
&x_2x_5 - x_4^2 + 2x_4x_5 - 3x_5^2 = 0.\\
\end{align*}

\noindent Genus of $X_0(83)$: $7$.

\noindent Cusps:
$(1 : 0 : 0 : 0 : 0 : 0 : 1), (-1 : 0 : 0 : 0 : 0 : 0 : 1)$.

\noindent $C = X_0^+(83)$: elliptic curve $y^2 + xy + y = x^3 + x^2 + x$.

\noindent Group structure of $J(C)(\Q)$: $J(C)(\Q) \simeq \Z \cdot [Q_C - O]$, where $Q_C := (0,0)$.

\noindent Group structure of $G \subseteq J_0(83)(\Q)$: $G \simeq  \Z \cdot D_{83} \oplus \Z/41\Z \cdot T_{83}$, where $D_{83}$, $T_{83}$ are as in \Cref{GandI}.

\noindent There are \textbf{no} exceptional non-cuspidal quadratic points on $X_0(83)$.

\noindent Primes used in sieve: $3$, $5$.

\subsection{$X_0(89)$}
Model for $X_0(89)$: \begin{align*} 
&x_0^2 - 2x_0x_1 - x_1^2 + x_2^2 + 6x_2x_3 - 21x_3^2 + 12x_3x_4 + 36x_3x_5 - 13x_4^2 - 6x_4x_5 - 21x_5^2 - x_6^2 = 0,\\
&x_0x_2 - x_1^2 + 3x_3^2 - 8x_3x_5 + x_4^2 + x_4x_5 + 5x_5^2 = 0,\\
&x_0x_3 - x_1x_2 + 3x_3^2 - 6x_3x_5 + 2x_4x_5 + 3x_5^2 = 0,\\
&x_0x_4 - x_2^2 + 2x_3^2 + 2x_3x_4 - 7x_3x_5 + 2x_4^2 - 2x_4x_5 + 5x_5^2 = 0,\\
&x_0x_5 - x_2x_3 + x_3^2 + 2x_3x_4 - 3x_3x_5 - x_4x_5 + 3x_5^2 = 0,\\
&x_1x_3 - x_2^2 + 3x_3x_4 - 4x_3x_5 + x_4^2 - 4x_4x_5 + 4x_5^2 = 0,\\
&x_1x_4 - x_2x_3 + 2x_3x_4 - x_3x_5 - 2x_4x_5 + x_5^2 = 0,\\
&x_1x_5 - x_3^2 + x_3x_4 + 2x_3x_5 - 2x_4x_5 - x_5^2 = 0,\\
&x_2x_4 - x_3^2 + 3x_3x_5 - x_4^2 - x_4x_5 - x_5^2 = 0,\\
&x_2x_5 - x_3x_4 + x_3x_5 + x_4x_5 - 2x_5^2 = 0.\\
\end{align*}

\noindent Genus of $X_0(89)$: $7$.

\noindent Cusps:
$(1 : 0 : 0 : 0 : 0 : 0 : 1), (-1 : 0 : 0 : 0 : 0 : 0 : 1)$.

\noindent $C = X_0^+(89)$: elliptic curve $y^2 - 19xy - y = x^3 - 89x^2 - 10x$.

\noindent Group structure of $J(C)(\Q)$: $J(C)(\Q) \simeq \Z \cdot [Q_C - O]$, where $Q_C := (0,0)$.

\noindent Group structure of $G \subseteq J_0(89)(\Q)$: $G \simeq  \Z \cdot D_{89} \oplus \Z/22\Z \cdot T_{89}$, where $D_{89}$, $T_{89}$ are as in \Cref{GandI}.

\noindent There are \textbf{no} exceptional non-cuspidal quadratic points on $X_0(89)$.

\noindent Primes used in sieve: $3$, $5$, $7$.
\newpage
\subsection{$X_0(95)$}
Model for $X_0(95)$: \begin{align*} 
&x_0^2 + 4x_2x_3 + 4x_2x_5 - 3x_3^2 + 2x_3x_4 + 4x_3x_5 + 19x_4^2 - 32x_4x_5 + 10x_5^2 - x_6^2 + 2x_7x_8 + 4x_8^2 = 0,\\
&x_0x_1 + 2x_1x_5 + 3x_2x_3 - 3x_2x_5 - 5x_3^2 - 2x_3x_4 + 6x_3x_5 + 14x_4^2 - 26x_4x_5 + 13x_5^2 - x_6x_7 + x_8^2 = 0,\\
&x_0x_2 - 2x_2x_5 - 2x_3^2 - x_3x_4 + x_3x_5 - 2x_4^2 + x_4x_5 + x_5^2 - x_7^2 = 0,\\
&x_0x_3 - 2x_2x_3 + 3x_3^2 - 2x_3x_5 - 9x_4^2 + 16x_4x_5 - 7x_5^2 - x_7x_8 - x_8^2 = 0,\\
&x_0x_4 - 2x_2x_3 + x_3^2 + 2x_3x_4 - 2x_3x_5 - 6x_4^2 + 10x_4x_5 - 4x_5^2 - x_8^2 = 0,\\
&x_0x_5 - x_2x_3 - x_2x_5 - x_3^2 + x_3x_4 + x_3x_5 - x_4^2 - x_4x_5 + 2x_5^2 = 0,\\
&x_0x_7 - x_1x_6 + x_4x_6 + 2x_4x_7 - x_4x_8 - x_5x_6 + x_5x_8 = 0,\\
&x_0x_8 - x_2x_6 + x_3x_6 - x_4x_7 + 3x_4x_8 + x_5x_6 + x_5x_7 - x_5x_8 = 0,\\
&x_1^2 - 4x_4^2 + 8x_4x_5 - 4x_5^2 - x_7^2 = 0,\\
&x_1x_2 - 2x_1x_5 - 2x_2x_3 + 2x_2x_5 + 4x_3^2 - 2x_3x_4 - 2x_3x_5 - 8x_4^2 + 18x_4x_5 - 10x_5^2 - x_7x_8 - x_8^2 = 0,\\
&x_1x_3 - x_1x_5 - 2x_2x_3 + 2x_2x_5 + 2x_3^2 + 2x_3x_4 - 2x_3x_5 - 8x_4^2 + 14x_4x_5 - 8x_5^2 - x_8^2 = 0,\\
&x_1x_4 - x_1x_5 - x_2x_3 + x_2x_5 + x_3^2 - 2x_4^2 + 4x_4x_5 - 3x_5^2 = 0,\\
&x_1x_7 - x_2x_6 - x_4x_7 + 2x_5x_6 + x_5x_7 = 0,\\
&x_1x_8 - x_3x_6 + x_4x_6 - x_4x_8 + x_5x_8 = 0,\\
&x_2^2 - 2x_2x_3 - 2x_2x_5 + x_3^2 + 4x_3x_4 - 2x_3x_5 - 8x_4^2 + 12x_4x_5 - 3x_5^2 - x_8^2 = 0,\\
&x_2x_4 - x_2x_5 - x_3^2 + x_3x_4 + x_3x_5 - 3x_4x_5 + 2x_5^2 = 0,\\
&x_2x_7 - x_3x_6 - x_4x_7 + x_5x_6 - x_5x_7 = 0,\\
&x_2x_8 - x_4x_6 - x_4x_8 + x_5x_6 - x_5x_8 = 0,\\
&x_3x_7 - x_4x_6 - x_4x_7 - x_4x_8 + x_5x_6 + x_5x_8 = 0,\\
&x_3x_8 - x_4x_7 + x_5x_7 - x_5x_8 = 0,\\
&x_6x_8 - x_7^2 + x_7x_8 + x_8^2 = 0.
\end{align*}

\noindent Genus of $X_0(95)$: $9$.

\noindent Cusps:
$(-1 : 0 : 0 : 0 : 0 : 0 : 1 : 0 : 0),
(1 : 0 : 0 : 0 : 0 : 0 : 1 : 0 : 0),
(-3/5 : 0 : -2/5 : -1/5 : -1/5 : -1/5 : 1 : 0 : 0),
(3/5 : 0 : 2/5 : 1/5 : 1/5 : 1/5 : 1 : 0 : 0)
$.

\noindent $C = X_0(95)/w_{19}$: hyperelliptic curve $y^2 = x^8 - 2x^7 - 7x^6 + 16x^5 - 2x^4 -2x^3 - 4x^2 + 5$.

\noindent Group structure of $J(C)(\Q)$: $J(C)(\Q) \simeq \Z/2\Z \cdot [P + P^{\sigma} - 2\infty_{+}] \oplus \Z/10\Z \cdot [\infty_{-} - \infty_{+}]$, where \[
P := \Big(\frac{1}{2}(-\sqrt{5} + 3) : \frac{1}{2}(5\sqrt{5} - 7) : 1 \Big) \in C(\Q(\sqrt{5})).
\]

\noindent Group structure of $G \subseteq J_0(95)(\Q)$: $G \simeq \Z/6\Z \cdot D_1 \oplus \Z/180\Z \cdot D_2$, where $D_1, D_2$ are generated by differences of cusps.

\noindent Below is the table of \textbf{all} quadratic points on $X_0(95)$ (up to Galois conjugacy) apart from cusps, which are all defined over $\Q$.

\begin{table}[h!]
\centering
 \begin{tabular}{||c c c c c||} 
 \hline
 Name & $\theta^2$ & Coordinates & $j$-invariant & CM \\ [0.5ex] 
 \hline\hline
 $P_1$ & -19 & $\scriptscriptstyle{(\frac{1}{14}(-\theta-17) : \frac{1}{7}(\theta-11) : \frac{1}{7}(\theta+3) : \frac{1}{7}(\theta-4) : \frac{1}{14}(\theta+3) : 1 : 0 : 0 : 0)}$ & -884736 & -19 \\ [1ex] 
 \hline
 \end{tabular}
\end{table}

\noindent Primes used in sieve: $11$, $13$.

\subsection{$X_0(101)$}
Model for $X_0(101)$:{\tiny
\begin{align*}
&x_0^2 - 4x_1^2 - 4x_1x_2 + 2x_2^2 + 8x_2x_3 + 4x_3^2 - 20x_3x_4 - 41x_4^2 + 112x_4x_5 + 18x_4x_6 - 22x_5^2 - 204x_5x_6 + 170x_6^2 - x_7^2 = 0,\\
&x_0x_2 - x_1^2 + 8x_3x_4 - 8x_4^2 - 15x_4x_5 + 21x_4x_6 + 32x_5^2 - 61x_5x_6 + 31x_6^2 = 0,\\
&x_0x_3 - x_1x_2 + 8x_3x_4 - 10x_4^2 - 5x_4x_5 + 17x_4x_6 + 22x_5^2 - 54x_5x_6 + 30x_6^2 = 0,\\
&x_0x_4 - x_2^2 + 6x_3x_4 - x_4^2 - 20x_4x_5 + 15x_4x_6 + 25x_5^2 - 30x_5x_6 + 9x_6^2 = 0,\\
&x_0x_5 - x_2x_3 + 4x_3x_4 - x_4^2 - 10x_4x_5 + 6x_4x_6 + 12x_5^2 - 8x_5x_6 = 0,\\
&x_0x_6 - x_3^2 + 2x_3x_4 - x_4x_5 - x_4x_6 - 2x_5^2 + 10x_5x_6 - 5x_6^2 = 0,\\
&x_1x_3 - x_2^2 + 8x_4^2 - 18x_4x_5 - 2x_4x_6 + 5x_5^2 + 26x_5x_6 - 23x_6^2 = 0,\\
&x_1x_4 - x_2x_3 + 6x_4^2 - 9x_4x_5 - 6x_4x_6 - 4x_5^2 + 33x_5x_6 - 24x_6^2 = 0,\\
&x_1x_5 - x_3^2 + 4x_4^2 - x_4x_5 - 8x_4x_6 - 10x_5^2 + 32x_5x_6 - 20x_6^2 = 0,\\
&x_1x_6 - x_3x_4 + 2x_4^2 - 3x_4x_6 - 4x_5^2 + 12x_5x_6 - 8x_6^2 = 0,\\
&x_2x_4 - x_3^2 + 6x_4x_5 - 7x_4x_6 - 10x_5^2 + 18x_5x_6 - 8x_6^2 = 0,\\
&x_2x_5 - x_3x_4 + 4x_4x_5 - 2x_4x_6 - 5x_5^2 + 3x_5x_6 = 0,\\
&x_2x_6 - x_4^2 + 2x_4x_5 - 5x_5x_6 + 3x_6^2 = 0,\\
&x_3x_5 - x_4^2 + 2x_4x_6 + 2x_5^2 - 9x_5x_6 + 6x_6^2 = 0,\\
&x_3x_6 - x_4x_5 + 2x_5^2 - 4x_5x_6 + 2x_6^2 = 0.\\
\end{align*}
}

\noindent Genus of $X_0(101)$: $8$.

\noindent Cusps:
$(1 : 0 : 0 : 0 : 0 : 0 : 0 : 1), (-1 : 0 : 0 : 0 : 0 : 0 : 0 : 1)$.

\noindent $C = X_0^+(101)$: elliptic curve $y^2 + y = x^3 + x^2 - x - 1$.

\noindent Group structure of $J(C)(\Q)$: $J(C)(\Q) \simeq \Z \cdot [Q_C - O]$, where $Q_C := (-1, 0)$.

\noindent Group structure of $G \subseteq J_0(101)(\Q)$: $G \simeq  \Z \cdot D_{101} \oplus \Z/25\Z \cdot T_{101}$, where $D_{101}$, $T_{101}$ are as in \Cref{GandI}.

\noindent There are \textbf{no} exceptional non-cuspidal quadratic points on $X_0(101)$.

\noindent Primes used in sieve: $3$, $5$.

\subsection{$X_0(119)$}
Model for $X_0(119)$:{\tiny
\begin{align*}
&x_0^2 + 8x_1x_6 - 8x_3x_4 + 24x_3x_6 + 17x_4^2 - 24x_4x_5 - 4x_4x_6 + 5x_5^2 + 38x_5x_6 - 39x_6^2 - x_7^2 + 2x_9^2 + 2x_9x_{10} + 6x_{10}^2 = 0,\\
&x_0x_1 + 4x_1x_6 - x_3x_4 + 2x_3x_6 + 3x_4x_5 - 4x_5^2 + 2x_5x_6 - 4x_6^2 - x_7x_8 + x_9x_{10} + x_{10}^2 = 0,\\
&x_0x_2 - 2x_1x_6 + 4x_2x_6 + 2x_3x_4 - 6x_3x_6 - 4x_4^2 + 4x_4x_5 + 2x_4x_6 + 3x_5^2 - 14x_5x_6 + 7x_6^2 - x_8^2 = 0,\\
&x_0x_3 - 3x_1x_6 + x_3x_4 - x_3x_6 - 3x_4^2 + 3x_4x_5 - 4x_5^2 + 2x_5x_6 - 4x_6^2 - x_8x_9 - x_9x_{10} = 0,\\
&x_0x_4 - 2x_1x_6 + 3x_3x_4 - 8x_3x_6 - 4x_4^2 + 5x_4x_5 + 2x_4x_6 + x_5^2 - 12x_5x_6 + 13x_6^2 - x_9^2 - 2x_{10}^2 = 0,\\
&x_0x_5 - x_1x_6 - x_3x_6 + x_4x_5 - 3x_4x_6 - x_5^2 + 4x_5x_6 + x_6^2 - x_9x_{10} = 0,\\
&x_0x_6 - x_1x_6 + x_3x_4 - 3x_3x_6 - 2x_4^2 + 3x_4x_5 - 6x_5x_6 + 8x_6^2 - x_{10}^2 = 0,\\
&x_0x_8 - x_1x_7 + x_5x_7 + 2x_5x_8 - x_6x_7 + 2x_6x_8 = 0,\\
&x_0x_9 - x_2x_7 + x_4x_7 + x_5x_8 + 4x_5x_9 - x_5x_{10} - x_6x_8 + x_6x_{10} = 0,\\
&x_0x_{10} - x_3x_7 + x_5x_7 - 2x_5x_8 + x_5x_9 + 4x_5x_{10} + 2x_6x_7 + 2x_6x_8 - x_6x_9 = 0,\\
\end{align*}
}
{\tiny
\begin{align*}
&x_1^2 + 4x_5^2 - 8x_5x_6 + 4x_6^2 - x_8^2 = 0,\\
&x_1x_2 - 2x_1x_6 - x_4^2 + x_4x_5 + 3x_4x_6 - 3x_5^2 + 4x_5x_6 - 5x_6^2 - x_8x_9 - x_9x_{10} = 0,\\
&x_1x_3 - 3x_1x_6 + 3x_3x_4 - 6x_3x_6 - 3x_4^2 + 2x_4x_5 + x_4x_6 + 3x_5^2 - 10x_5x_6 + 13x_6^2 - x_9^2 - 2x_{10}^2 = 0,\\
&x_1x_4 - 2x_1x_6 + x_4^2 - x_4x_5 - 3x_4x_6 - x_5^2 + 4x_5x_6 + x_6^2 - x_9x_{10} = 0,\\
&x_1x_5 - x_1x_6 + x_3x_4 - 2x_3x_6 - x_4^2 + 2x_4x_5 - x_4x_6 + x_5^2 - 6x_5x_6 + 7x_6^2 - x_{10}^2 = 0,\\
&x_1x_8 - x_2x_7 - x_5x_8 + 2x_6x_7 + x_6x_8 = 0,\\
&x_1x_9 - x_3x_7 + x_5x_7 - 2x_5x_8 + 2x_5x_{10} + 2x_6x_7 + 2x_6x_8 - 2x_6x_{10} = 0,\\
&x_1x_{10} - x_4x_7 - x_5x_8 - 2x_5x_9 + 2x_6x_7 + x_6x_8 + 2x_6x_9 = 0,\\
&x_2^2 - 4x_2x_6 + x_4^2 - 2x_4x_5 - 2x_4x_6 + 2x_5^2 + 6x_6^2 - x_9^2 - 2x_{10}^2 = 0,\\
&x_2x_3 - 3x_2x_6 - 2x_3x_4 + 2x_3x_6 + 3x_4^2 - 2x_4x_5 - 4x_4x_6 - 3x_5^2 + 10x_5x_6 - x_6^2 - x_9x_{10} = 0,\\
&x_2x_4 - 2x_2x_6 + x_4x_5 - 3x_4x_6 + x_5^2 - 4x_5x_6 + 7x_6^2 - x_{10}^2 = 0,\\
&x_2x_5 - x_2x_6 - x_3x_4 + 2x_3x_6 + x_4^2 - x_4x_5 - x_5^2 + 2x_5x_6 - 3x_6^2 = 0,\\
&x_2x_8 - x_3x_7 - 3x_5x_8 + 3x_6x_7 + x_6x_8 = 0,\\
&x_2x_9 - x_4x_7 - x_5x_8 - 2x_5x_9 + x_5x_{10} + 2x_6x_7 + x_6x_8 - x_6x_{10} = 0,\\
&x_2x_{10} - x_5x_7 - x_5x_8 - x_5x_9 - 2x_5x_{10} + x_6x_7 + x_6x_8 + x_6x_9 = 0,\\
&x_3^2 - 6x_3x_6 - 2x_4^2 + 6x_4x_5 + 2x_4x_6 - 3x_5^2 - 6x_5x_6 + 10x_6^2 - x_{10}^2 = 0,\\
&x_3x_5 - x_3x_6 - x_4^2 + x_4x_5 + 3x_4x_6 - 2x_5^2 - x_5x_6 - x_6^2 = 0,\\
&x_3x_8 - x_4x_7 - 2x_5x_8 + 2x_6x_7 - x_6x_8 = 0,\\
&x_3x_9 - x_5x_7 - x_5x_8 - x_5x_9 + x_5x_{10} + x_6x_7 + x_6x_8 - 2x_6x_9 - x_6x_{10} = 0,\\
&x_3x_{10} - x_5x_8 - x_5x_9 - x_5x_{10} + x_6x_8 + x_6x_9 - 2x_6x_{10} = 0,\\
&x_4x_8 - x_5x_7 - x_5x_8 + x_6x_7 - x_6x_8 = 0,\\
&x_4x_9 - x_5x_8 - x_5x_9 + x_5x_{10} + x_6x_8 - x_6x_9 - x_6x_{10} = 0,\\
&x_4x_{10} - x_5x_9 - x_5x_{10} + x_6x_9 - x_6x_{10} = 0,\\
&x_7x_9 - x_8^2 + x_9^2 + x_{10}^2 = 0,\\
&x_7x_{10} - x_8x_9 = 0,\\
&x_8x_{10} - x_9^2 - x_{10}^2 = 0.
\end{align*}
}%

\noindent Genus of $X_0(119)$: $11$.

\noindent Cusps:
$(-1 : 0 : 0 : 0 : 0 : 0 : 0 : 1 : 0 : 0 : 0),
(1 : 0 : 0 : 0 : 0 : 0 : 0 : 1 : 0 : 0 : 0),
(-3/7 : 0 : -2/7 : -3/7 : -2/7 : -1/7 : -1/7 : 1 : 0 : 0 : 0),
(3/7 : 0 : 2/7 : 3/7 : 2/7 : 1/7 : 1/7 : 1 : 0 : 0 : 0)$.

\noindent $C = X_0(119)/w_{17}$: hyperelliptic curve $y^2 = x^{10} + 2x^8 - 11x^6 + 14x^5 - 40x^4 + 42x^3 - 48x^2 + 28x - 7$.

\noindent Group structure of $J(C)(\Q)$: $J(C)(\Q) \simeq \Z/9\Z \cdot [\infty_{+} - \infty_{-}]$.

\noindent Group structure of $G \subseteq J_0(119)(\Q)$: $G \simeq \Z/8\Z \cdot D_1 \oplus \Z/288\Z \cdot D_2$, where $D_1, D_2$ are generated by differences of cusps.

\noindent Below is the table of \textbf{all} quadratic points on $X_0(119)$ (up to Galois conjugacy) apart from cusps which are all defined over $\Q$.

\begin{table}[h!]
\centering
 \begin{tabular}{||c c c c c||} 
 \hline
 Name & $\theta^2$ & Coordinates & $j$-invariant & CM \\ [0.5ex] 
 \hline\hline
 $P_1$ & -19 & $\scriptscriptstyle{(\frac{1}{7}(-2\theta + 1) : 0 : \frac{1}{7}(\theta - 4) : \frac{1}{14}(3\theta - 19) : \frac{1}{7}(\theta + 3) : \frac{1}{14}(\theta - 11) : \frac{1}{14}(\theta + 3) : -2 : 2 : -1 : 1)}$ & -884736 & -19 \\ [1ex] 
 \hline
 $P_2$ & -19 & $\scriptscriptstyle{(\frac{1}{7}(-2\theta + 1) : 0 : \frac{1}{7}(\theta - 4) : \frac{1}{14}(3\theta - 19) : \frac{1}{7}(\theta + 3) : \frac{1}{14}(\theta - 11) : \frac{1}{14}(\theta + 3) : 2 : -2 : 1 : -1)}$ & -884736 & -19\\ [1ex] 
 \hline
 \end{tabular}
\end{table}

\noindent Primes used in sieve: $5$.

\subsection{$X_0(131)$}
Model for $X_0(131)$:{\tiny
\begin{align*}
&x_0^2 - 2x_1^2 - 4x_1x_2 - 3x_2^2 + 4x_2x_3 + 6x_3^2 + 8x_3x_4 + 2x_4^2 - 48x_4x_5 - 44x_5^2 + 38x_5x_6 + 337x_6^2 - 244x_6x_7 - 1368x_7^2 + \\
&3738x_7x_8 - 4056x_7x_9 - 4088x_8^2 + 6808x_8x_9 - 2706x_9^2 - x_{10}^2 = 0,\\
&x_0x_2 - x_1^2 + 12x_4x_5 - 2x_5^2 - 36x_5x_6 - 3x_6^2 + 119x_6x_7 - 108x_7^2 + 38x_7x_8 - 12x_7x_9 + 22x_8^2 - 8x_8x_9 + 98x_9^2 = 0,\\
&x_0x_3 - x_1x_2 + 14x_4x_5 - 12x_5^2 - 14x_5x_6 - 15x_6^2 + 113x_6x_7 - 106x_7^2 + 31x_7x_8 - 20x_7x_9 + 32x_8^2 - 3x_8x_9 + 86x_9^2 = 0,\\
&x_0x_4 - x_2^2 + 12x_4x_5 - 35x_5x_6 + x_6^2 + 92x_6x_7 - 71x_7^2 + 4x_7x_8 + 21x_7x_9 + 42x_8^2 - 54x_8x_9 + 93x_9^2 = 0,\\
&x_0x_5 - x_2x_3 + 9x_4x_5 - 22x_5x_6 - 12x_6^2 + 87x_6x_7 - 43x_7^2 - 59x_7x_8 + 84x_7x_9 + 109x_8^2 - 162x_8x_9 + 128x_9^2 = 0,\\
&x_0x_6 - x_3^2 + 6x_4x_5 - 5x_5x_6 - 27x_6^2 + 48x_6x_7 + 67x_7^2 - 255x_7x_8 + 280x_7x_9 + 297x_8^2 - 480x_8x_9 + 214x_9^2 = 0,\\
&x_0x_7 - x_3x_4 + 4x_4x_5 - 4x_5x_6 - 12x_6^2 + 21x_6x_7 + 35x_7^2 - 116x_7x_8 + 128x_7x_9 + 134x_8^2 - 218x_8x_9 + 94x_9^2 = 0,\\
&x_0x_8 - x_4^2 + 2x_4x_5 - x_5x_6 + x_6^2 - 8x_6x_7 + 18x_7^2 - 25x_7x_8 + 29x_7x_9 + 29x_8^2 - 45x_8x_9 + 11x_9^2 = 0,\\
&x_0x_9 - x_5^2 + 2x_5x_6 + x_6^2 - 2x_6x_7 - 7x_7^2 + 20x_7x_8 - 16x_7x_9 - 20x_8^2 + 37x_8x_9 - 10x_9^2 = 0,\\
&x_1x_3 - x_2^2 + 14x_5^2 - 24x_5x_6 - 12x_6^2 + 21x_6x_7 + 98x_7^2 - 261x_7x_8 + 298x_7x_9 + 281x_8^2 - 492x_8x_9 + 202x_9^2 = 0,\\
&x_1x_4 - x_2x_3 + 12x_5^2 - 14x_5x_6 - 23x_6^2 + 15x_6x_7 + 131x_7^2 - 334x_7x_8 + 370x_7x_9 + 357x_8^2 - 613x_8x_9 + 240x_9^2 = 0,\\
&x_1x_5 - x_3^2 + 9x_5^2 - 36x_6^2 + 182x_7^2 - 436x_7x_8 + 468x_7x_9 + 460x_8^2 - 777x_8x_9 + 287x_9^2 = 0,\\
&x_1x_6 - x_3x_4 + 6x_5^2 - 17x_6^2 - 13x_6x_7 + 113x_7^2 - 244x_7x_8 + 261x_7x_9 + 249x_8^2 - 428x_8x_9 + 148x_9^2 = 0,\\
&x_1x_7 - x_4^2 + 4x_5^2 - 4x_6^2 - 24x_6x_7 + 71x_7^2 - 122x_7x_8 + 129x_7x_9 + 115x_8^2 - 205x_8x_9 + 59x_9^2 = 0,\\
&x_1x_8 - x_4x_5 + 2x_5^2 - x_6^2 - 8x_6x_7 + 22x_7^2 - 36x_7x_8 + 39x_7x_9 + 33x_8^2 - 62x_8x_9 + 17x_9^2 = 0,\\
&x_1x_9 - x_5x_6 + 2x_6^2 + x_6x_7 - 8x_7^2 + 18x_7x_8 - 19x_7x_9 - 19x_8^2 + 31x_8x_9 - 12x_9^2 = 0,\\
&x_2x_4 - x_3^2 + 12x_5x_6 - 28x_6^2 + x_6x_7 + 99x_7^2 - 246x_7x_8 + 258x_7x_9 + 264x_8^2 - 434x_8x_9 + 160x_9^2 = 0,\\
&x_2x_5 - x_3x_4 + 9x_5x_6 - 12x_6^2 - 22x_6x_7 + 77x_7^2 - 146x_7x_8 + 149x_7x_9 + 144x_8^2 - 243x_8x_9 + 72x_9^2 = 0,\\
&x_2x_6 - x_4^2 + 6x_5x_6 - 29x_6x_7 + 37x_7^2 - 34x_7x_8 + 30x_7x_9 + 21x_8^2 - 42x_8x_9 - 7x_9^2 = 0,\\
&x_2x_7 - x_4x_5 + 4x_5x_6 - 13x_6x_7 + 12x_7^2 - 9x_7x_8 + 6x_7x_9 + 3x_8^2 - 7x_8x_9 - 7x_9^2 = 0,\\
&x_2x_8 - x_5^2 + 2x_5x_6 - x_6x_7 - 5x_7^2 + 12x_7x_8 - 16x_7x_9 - 16x_8^2 + 26x_8x_9 - 11x_9^2 = 0,\\
&x_2x_9 - x_6^2 + 2x_6x_7 + x_7^2 - 6x_7x_8 + 4x_7x_9 + 7x_8^2 - 12x_8x_9 + 4x_9^2 = 0,\\
&x_3x_5 - x_4^2 + 9x_6^2 - 24x_6x_7 + 6x_7^2 + 34x_7x_8 - 39x_7x_9 - 50x_8^2 + 74x_8x_9 - 45x_9^2 = 0,\\
&x_3x_6 - x_4x_5 + 6x_6^2 - 9x_6x_7 - 17x_7^2 + 59x_7x_8 - 64x_7x_9 - 67x_8^2 + 109x_8x_9 - 47x_9^2 = 0,\\
&x_3x_7 - x_5^2 + 4x_6^2 - 22x_7^2 + 48x_7x_8 - 52x_7x_9 - 51x_8^2 + 86x_8x_9 - 31x_9^2 = 0,\\
&x_3x_8 - x_5x_6 + 2x_6^2 - 7x_7^2 + 16x_7x_8 - 19x_7x_9 - 21x_8^2 + 32x_8x_9 - 12x_9^2 = 0,\\
&x_3x_9 - x_6x_7 + 2x_7^2 - 3x_7x_8 - x_7x_9 + 2x_8^2 - 4x_8x_9 - 2x_9^2 = 0,\\
&x_4x_6 - x_5^2 + 6x_6x_7 - 18x_7^2 + 31x_7x_8 - 33x_7x_9 - 29x_8^2 + 52x_8x_9 - 15x_9^2 = 0,\\
&x_4x_7 - x_5x_6 + 4x_6x_7 - 6x_7^2 + 3x_7x_8 - 2x_7x_9 - x_8^2 + 2x_8x_9 + 4x_9^2 = 0,\\
&x_4x_8 - x_6^2 + 2x_6x_7 - 5x_7x_8 + 4x_7x_9 + 3x_8^2 - 7x_8x_9 + 4x_9^2 = 0,\\
&x_4x_9 - x_7^2 + 2x_7x_8 - 6x_7x_9 - 3x_8^2 + 5x_8x_9 - 5x_9^2 = 0,\\
&x_5x_7 - x_6^2 + 4x_7^2 - 12x_7x_8 + 14x_7x_9 + 13x_8^2 - 23x_8x_9 + 10x_9^2 = 0,\\
&x_5x_8 - x_6x_7 + 2x_7^2 - 4x_7x_8 + 2x_7x_9 + x_8^2 - 3x_8x_9 = 0,\\
&x_5x_9 - x_7x_8 - 3x_7x_9 + x_8x_9 - 3x_9^2 = 0,\\
&x_6x_8 - x_7^2 + 2x_7x_8 - 3x_7x_9 - 4x_8^2 + 5x_8x_9 - 2x_9^2 = 0,\\
&x_6x_9 - 2x_7x_9 - x_8^2 + x_8x_9 - 2x_9^2 = 0.\\
\end{align*}
}

\noindent Genus of $X_0(131)$: $11$.

\noindent Cusps:
$(1 : 0 : 0 : 0 : 0 : 0 : 0 : 0 : 0 : 0 : 1), (-1 : 0 : 0 : 0 : 0 : 0 : 0 : 0 : 0 : 0 : 1)$.

\noindent $C = X_0^+(131)$: elliptic curve $y^2 + y = x^3 - x^2 + x$ %(not explicitly used).

\noindent Group structure of $J(C)(\Q)$: $J(C)(\Q) \simeq \Z \cdot [Q_C - O]$, where $Q_C := (0, 0)$.

\noindent Group structure of $G \subseteq J_0(131)(\Q)$: $G \simeq \Z/65\Z \cdot T_{131}$, where $T_{131}$ is as in \Cref{GandI}.

\noindent There are \textbf{no} exceptional non-cuspidal quadratic points on $X_0(131)$.

\noindent Primes used in sieve: $3$, $5$.

\bibliographystyle{siam}
\bibliography{bibliography1}
\end{document}